\documentclass[a4paper,oneside,12pt]{article}
\usepackage[utf8]{inputenc}
\usepackage{lmodern}
\usepackage[T1]{fontenc}
\usepackage{textcomp}
\usepackage[english]{babel}
\usepackage[a4paper,vmargin={2.5cm,3.5cm},hmargin=2cm]{geometry}
\usepackage[pdftex]{hyperref} 
\usepackage[expansion]{microtype}
\usepackage[pdftex]{color,graphicx}

\usepackage{amsmath,amsfonts,amssymb,amsthm,mathrsfs}

\theoremstyle{plain}
\newtheorem{theorem}{Theorem}
\newtheorem{corollary}{Corollary}
\newtheorem{proposition}{Proposition}

\newtheorem{lemma}{Lemma}
\theoremstyle{definition}

\theoremstyle{remark}
\newtheorem*{remark}{Remark}

\def\build#1_#2^#3{\mathrel{ \mathop{\kern 0pt#1}\limits_{#2}^{#3}}}

\hypersetup{
    pdfauthor   = {Jean-Fran\c cois Le Gall, Laurent Menard},%
    pdftitle    = {Scaling limits for the uniform infinite quadrangulation}}

\author{Jean-Fran\c cois {\sc Le Gall}\footnote{\texttt{jean-francois.legall@math.u-psud.fr}} \qquad Laurent {\sc M\'enard}\footnote{\texttt{laurent.menard@normalesup.org}}\\
        D\'epartement de Math\'ematiques\\
        Universit\'e Paris-Sud\\
        91405 Orsay Cedex, France}
\title{\bf SCALING LIMITS FOR THE UNIFORM INFINITE QUADRANGULATION}
\date{{\small April 2010}}

\begin{document}

\maketitle

\begin{abstract}
The uniform infinite planar quadrangulation is an infinite random graph embedded in the plane, which is the local limit of uniformly
distributed finite quadrangulations with a fixed number of faces.
We study asymptotic properties of this random graph. In particular, we investigate scaling limits of the profile of distances from the distinguished point called the root, and we get asymptotics for the volume of large balls. As a key technical tool, we first describe the scaling limit of the contour functions of the uniform infinite well-labeled tree, in terms of a pair of eternal conditioned Brownian snakes. Scaling limits for the uniform infinite quadrangulation can then be derived thanks to an extended version of Schaeffer's bijection between well-labeled trees and rooted quadrangulations.
\end{abstract}

\section{Introduction}
\label{sec:introduction}

The main purpose of the present work is to study asymptotic properties of
the infinite random graph called the uniform infinite quadrangulation.
Recall that planar maps are proper embeddings of finite connected graphs in the two-dimensional sphere, considered 
up to orientation-preserving homeomorphisms of the sphere. It is convenient to deal with rooted maps, meaning that
there is a distinguished oriented edge, whose origin is called the root vertex. Given a planar map, its faces are
the regions delimited by the edges.
Important special cases of planar maps are
triangulations, respectively quadrangulations, where each face of the map is adjacent to three edges, resp. to four edges.

Combinatorial properties of planar maps have been studied extensively since the work of Tutte \cite{Tutte}, which
was motivated by the famous four color theorem. Planar maps have also been considered in the theoretical physics literature because of their connections with matrix integrals (see \cite{BIPZ}). More recently,
they have been used in physics as models of random surfaces, especially in the setting of the theory of two-dimensional quantum gravity (see in particular the book by Ambj{\o}rn, Durhuus and Jonsson \cite{2DQG}). 

In a pioneering paper, Angel and Schramm \cite{AS} defined an infinite random triangulation of the plane, whose law is uniform in the sense that it is the local limit of uniformly distributed triangulations with a fixed number of faces, when this number tends to infinity. Various properties of the uniform infinite triangulation, including the study of percolation on this infinite random graph, were
derived by Angel \cite{A1} (see also Krikun \cite{Krikun}). Some intriguing questions, such as the recurrence of random walk on the
uniform infinite triangulation, still remain open.

Although quadrangulations may seem to be more complicated objects 
than triangulations, some of their properties can be studied more easily
because they are bipartite graphs, and especially thanks to the existence
of a remarkable bijection between the set of all (rooted) quadrangulations with a fixed number of faces and the set of all well-labeled trees with the same number of edges. See 
\cite{CS} for a thorough discussion of this correspondence, which we call 
Schaeffer's bijection. Motivated by this bijection, 
Chassaing and Durhuus \cite{CD} constructed
the so-called uniform infinite well-labeled tree, and then
used an extended version of Schaeffer's bijection to get an infinite random quadrangulation from this infinite random tree.
A little later, Krikun \cite{Kr} constructed the  uniform infinite quadrangulation as the local limit of uniform finite quadrangulations as their size goes to infinity, in the spirit of the work of Angel and Schramm for triangulations. It was proved in \cite{Me} that both these constructions lead to
the same infinite random graph, which is the object of interest in the present work. 

Before describing our main results, let us recall the definition of the 
uniform infinite well-labeled tree. A (finite) well-labeled tree is 
a rooted ordered tree whose vertices are assigned positive integer labels, 
in such a way that the root has label one, and the labels of two
neighboring vertices can differ by at most one in absolute value. 
Chassaing and Durhuus \cite{CD} showed that the uniform probability 
distribution on the set of all well-labeled trees with $n$ edges 
converges as $n\to\infty$ towards a probability measure $\mu$ supported on 
infinite well-labeled trees, which is called the law of the
uniform infinite well-labeled tree. It was also proved in \cite{CD}
that an infinite tree distributed according to $\mu$ has a.s.
a unique spine, that is a unique infinite injective path starting from the root.

Thanks to this property, the uniform infinite well-labeled tree can be coded by two pairs of contour functions $( C^{(L)},V^{(L)} )$ and $( C^{(R)},V^{(R)} )$ 
corresponding respectively to the left side and the right side of the spine. Roughly speaking (see subsect. \ref{subsubsec:spatial trees} for more precise definitions),
if we imagine a particle that explores the left side of the spine by traversing the tree from the left to
the right, then for every integer $k$, $C^{(L)}_k$ is the height in the tree 
of the vertex visited by the particle at time $k$, and $V^{(L)}_k$ is the label of the
same vertex. The pair  $( C^{(R)},V^{(R)})$  is defined analogously
for the right side of the spine. We obtain asymptotics for the uniform infinite well-labeled tree in the 
form of the following convergence in distribution (Theorem \ref{snakeinfinite}):
\begin{align}
\label{maincd}
& \Big( \Big( \frac{1}{n} C^{(L)} (n^2 t) , \sqrt{\frac{3}{2n}} V^{(L)}(n^2 t) \Big)_{t \geq 0},
\Big( \frac{1}{n} C^{(R)}(n^2 t) , \sqrt{\frac{3}{2n}} V^{(R)}(n^2 t) \Big)_{t \geq 0} \Big)
\notag \\
& \qquad \qquad \underset{n\to\infty}{\longrightarrow}
\Big( \Big(\zeta^{(L)}_t,\widehat{W}^{(L)}_t \Big)_{t\geq 0},
\Big(\zeta^{(R)}_t,\widehat{W}^{(R)}_t \Big)_{t\geq 0} \Big).
\end{align}
Here $\zeta^{(L)}$ and $\widehat{W}^{(L)}$ represent respectively the lifetime process and 
the endpoint process of a path-valued process $W^{(L)}$ called the eternal conditioned
Brownian snake. Roughly speaking,  the eternal conditioned
Brownian snake should be interpreted as a one-dimensional Brownian snake
started from $0$
(see \cite{LGsnake}) and conditioned not to hit the negative half-line. 
This process was introduced in \cite{LGW}, where it was shown to
be the limit in distribution of a Brownian snake driven by a Brownian excursion
and conditioned to stay positive, when the height of the excursion tends 
to infinity (see Theorem 4.3 in \cite{LGW}). Similarly the pair $(\zeta^{(R)},\widehat{W}^{(R)})$ is obtained from another eternal conditioned
Brownian snake $W^{(R)}$. Note however that the processes 
$W^{(L)}$ and $W^{(R)}$ are not independent: The dependence 
between $W^{(L)}$ and $W^{(R)}$ comes from
the labels on the spine, which are (of course) the same when 
exploring the left side and the right side of the tree. 

We can combine the convergence \eqref{maincd} with the extended version of Schaeffer's bijection in order to derive asymptotics for distances in
the uniform infinite quadrangulation in terms of the eternal conditioned Brownian snake.
Here we use a key property of Schaeffer's bijection, which remains valid in
the infinite setting: If a quadrangulation is asociated with a well-labeled tree
in this bijection, vertices of the quadrangulation (except the root vertex) exactly 
correspond to vertices of the tree, and the graph distance in the quadrangulation
between a vertex $v$ and the root coincides with the label of $v$ on the tree. If 
$V({\bf q})$ stands for the set of vertices of the uniform infinite quadrangulation ${\bf q}$
and if $d_{gr}(\partial,v)$ denotes the graph distance between vertex $v$ and
the root vertex $\partial$, we let the profile of distances be the
$\sigma$-finite measure on ${\mathbb Z}_+$ defined by
$$\lambda_{\bf q}(k)=\#\{ v\in V({\bf q}): d_{gr}(\partial,v)=k\},$$
for every $k\in {\mathbb Z}_+$.
For every integer $n\geq 1$, we also define a rescaled profile $\lambda^{(n)}_{\bf q}$ by
$$\lambda^{(n)}_{\bf q}(A)= n^{-2}\lambda\Big(\sqrt{\frac{2n}{3}}\,A\Big),$$
for every Borel subset $A$ of $\mathbb R_+$. Then Theorem \ref{profile}
shows that the sequence $\lambda^{(n)}_{\bf q}$ converges in distribution
towards the random measure ${\mathcal I}$ defined by 
$$\langle \mathcal{I},g \rangle = \frac{1}{2}\int_0^{\infty} \mathrm{d}s
\left(g \left( \widehat{W}_s^{(L)} \right) + g \left( \widehat{W}_s^{(R)} \right)\right)$$
for every continuous function $g$ with compact support on ${\mathbb R}_+$.
As a consequence, if $B_{n}({\bf q})$ denotes the ball of radius $n$
centered at $\partial$ in $V({\bf q})$, we also get the convergence in
distribution of $n^{-4}\# B_{n}({\bf q})$ as $n\to\infty$. 

Although the present work concentrates on the profile 
of distances, we expect that the convergence \eqref{maincd} will 
have applications
to other problems concerning the uniform infinite quadrangulation
and random walk on this graph (similarly as in the
case of the uniform infinite triangulation, the recurrence of this random
walk is still an open question). Indeed, thanks to the explicit construction 
of edges of the map from the associated tree in Schaeffer's bijection, scaling limits for
the uniform infinite well-labelled tree 
should lead to useful information about the geometry of the 
uniform infinite quadrangulation. We hope to address these questions 
in some future work.

To conclude this introduction, let us mention that
a different approach to asymptotics for large planar maps has been 
developed in several recent papers, which do not deal with local limits 
but instead study the convergence of rescaled random planar
maps viewed as random compact metric spaces, in the sense of the Gromov-Hausdorff
distance. 
In particular, the paper \cite{LG} proves that,
at least along suitable sequences, uniformly distributed quadrangulations
with $n$ faces,
equipped with the graph distance rescaled by the factor $n^{-1/4}$ and
viewed as random metric spaces, converge in distribution in the sense of the Gromov-Hausdorff distance towards the so-called Brownian map. The Brownian map is a quotient space of Aldous' continuum random tree \cite{Aldous} for an equivalence relation defined in terms of Brownian labels assigned to the vertices of the tree. It was first introduced by Marckert and Mokkadem \cite{MaMo},
who obtained a weak form of the convergence of rescaled quadrangulations towards the Brownian map. Although
we do not pursue this matter here, we note that 
the limiting process appearing in the convergence (\ref{maincd}) should 
play a role
in the study of the Brownian map, and should indeed be related to the geometry of the
Brownian map near a typical point. We may also observe that the convergence (\ref{maincd}) is 
an infinite tree version of the main theorem of \cite{LG06}, which gives the
scaling limit of the contour functions of well-labeled trees with a (large) fixed number of edges
and plays a crucial role in the convergence of rescaled quadrangulations towards
the Brownian map.

The paper is organized as follows. Section \ref{sec: b} contains preliminaries about trees, finite or infinite quadrangulations, and the extended version of Schaeffer's bijection. We also discuss the uniform infinite well-labeled tree and quadrangulation as defined in \cite{CD,Kr} and recall some basic facts about the Brownian snake. Section \ref{sec:scaletree}
contains the most technical part of this work, which is the proof of the convergence \eqref{maincd}. Our applications to
scaling limits for the uniform infinite quadrangulation are discussed in Section \ref{sec:scalequad}.

\smallskip
{\it Notation}. If $I$ is an interval of the real line, and $E$ is a metric space, the notation $C(I,E)$
stands for the space of all continuous functions from $I$ into $E$. This space is equipped with
the topology of uniform convergence on compact sets. If $E$ is a Polish space, ${\mathbb D}(E)$
stands for the space of all c\`adl\`ag functions from $[0,\infty[$ into $E$, which is
equipped with the usual Skorokhod topology. 

\section{Preliminaries}
\label{sec: b}

\subsection{Trees and quadrangulations}
\label{subsec:trees quadrangulations}

\subsubsection{Spatial trees}
\label{subsubsec:spatial trees}

In order to give precise definitions of the objects of interest in this work, it will be convenient to use the standard formalism for plane trees. Let
\[
\mathcal{U} = \bigcup_{n=0}^{\infty} \mathbb{N}^n
\]
where $\mathbb{N} = \{ 1,2, \ldots \}$ and $\mathbb{N}^0 = \{ \emptyset \}$ by convention. An element $u$ of $\mathcal{U}$ is thus a finite sequence 
$u=(u_1,\ldots,u_n)$ of positive integers, and $n={\rm gen}(u)$ is called the {\it generation} of $u$. If $u, v \in \mathcal{U}$, $uv$ denotes the concatenation of $u$ and $v$. If $v$ is of the form $uj$ with $j \in \mathbb{N}$, we say that $u$ is the \emph{parent} of $v$ or that $v$ is a \emph{child} of $u$. We use the notation $v\prec v'$ for the (strict) lexicographical order on $\mathcal U$.

A \emph{plane tree} $\tau$ is a (finite or infinite) subset of $\mathcal{U}$ such that
\begin{enumerate}
\item $\emptyset \in \tau$  ($\emptyset$ is called the \emph{root} of $\tau$),
\item if $v \in \tau$ and $v \neq \emptyset$, the parent of $v$ belongs to $\tau$
\item for every $u \in \mathcal{U}$ there exists an integer $k_u(\tau) \geq 0$ such that, for every $j\in\mathbb{N}$, $uj \in \tau$ if and only if $j \leq k_u(\tau)$.
\end{enumerate}
The edges of $\tau$ are the pairs $(u,v)$, where $u, v \in \tau$ and $u$ is the parent of $v$. The integer $|\tau|$ denotes the number of edges of $\tau$ and is called the size of $\tau$. The height $H(\tau)$ of $\tau$ is defined by $H(\tau)=\sup\{{\rm gen}(u):u\in\tau\}$.
A \emph{spine}
of $\tau$ is an infinite linear subtree of $\tau$ starting
from its root (of course a spine can only exist if $\tau$ is infinite). We denote by $\mathcal{T}$ the set of all plane trees. 

\smallskip

A \emph{labeled tree} (or spatial tree) is a pair $\theta = (\tau, (\ell(u))_{u \in \tau})$  that consists of a plane tree $\tau$ and a collection of integer labels assigned to the vertices of $\tau$, such that if $(u,v)$ is an edge of $\tau$, then $|\ell(u) - \ell(v)| \leq 1$.

A labeled tree $(\tau, (\ell(u))_{u \in \tau})$ such that $\ell(\emptyset) = 1$ and $\ell(u) \geq 1$ for every $u\in\tau$
is called a well-labeled tree. We denote the space of all 
well-labeled trees by $\overline{\mathbf T}$. The notation  $\mathbf{T}$, respectively $\mathbf{T}_{\infty}$, resp.  $\mathbf{T}_n$, will stand for the set of all well-labeled trees  that have finitely many edges, resp.  infinitely many edges, resp. $n$ edges. 

 If $\theta = (\tau, (\ell(u))_{u \in \tau})$  is a labeled tree, $|\theta| = |\tau|$ is the size of $\theta$ and $H(\theta) = H(\tau)$ is the height of $\theta$.
 A spine of $\theta$ is a spine of $\tau$.

\smallskip

A finite labeled tree $\theta = (\tau, \ell)$ can be coded by a pair $\left( C_{\theta},V_{\theta} \right)$, where $C_{\theta} = \left( C_{\theta} (t) \right)_{0 \leq t \leq 2|\theta|}$ is the contour function of $\tau$ and $V_{\theta} = \left( V_{\theta} (t) \right)_{0 \leq t \leq 2|\theta|}$ is the spatial contour function of $\theta$ (see Fig. \ref{fig:contour}). To define these contour functions, let us consider a particle which follows the contour of the tree from the
left to the right, in the following sense. The particle starts from the root and traverses the tree along its edges at speed one. When leaving a vertex, the particle moves towards the first non visited child of this vertex if there is such a child, or returns to the parent of this vertex. Since all edges will be crossed twice, the total time needed to explore the tree is $2 |\theta|$. For every $t \in [0, 2|\theta|]$, $C_{\theta}(t)$ denotes the distance from the root of the position of the particle at time $t$. In addition if $t \in [0, 2 |\theta|]$ is an integer, $V_{\theta}(t)$ denotes the label of the vertex that is visited at time $t$. We then complete the definition of $V_{\theta}$ by interpolating linearly between successive integers.  Fig. 1 explains
the construction of the contour functions better than a formal definition.

A finite labeled tree is uniquely determined by its pair of contour functions. It will sometimes be convenient to
define the functions $C_\theta$ and $V_\theta$ for every $t\geq 0$, by setting $C_\theta(t)=0$ and $V_\theta(t)=V_\theta(0)$ for every 
$t\geq 2|\theta|$. 

\begin{figure}[!ht]
\begin{center}
\includegraphics[width=0.8 \textwidth]{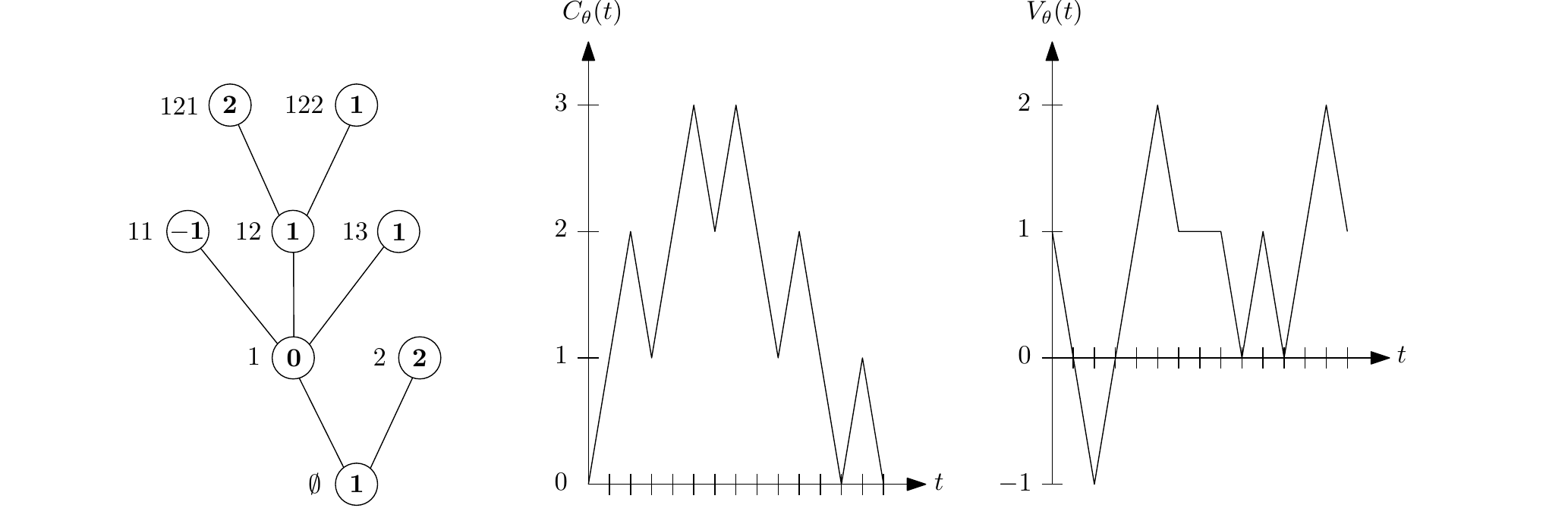}
\end{center}
\caption{A labeled tree $\theta$ and its pair of contour functions $\left( C_{\theta} ,V_{\theta} \right)$.}
\label{fig:contour}
\end{figure}

\smallskip

If $\theta$ and $ \theta' $ are two labeled trees, we define
\[
d(\theta,\theta') = \left( 1 + \sup \left\{ h: \, {\rm tr}_{h}(\theta) = {\rm tr}_{h}(\theta') \right\} \right)^{-1}
\]
where, for every integer $h \geq 0$, ${\rm tr}_{h}(\theta)$ is the labeled tree consisting of all vertices of $\theta$ up to generation $h$, with the same labels. One easily checks that 
$d$ is a distance on the space of all labeled trees.

\smallskip

If $\theta\in\overline{\mathbf{T}}$, for every $k \in \mathbb{N}$, we let $N_k(\theta)$ denote the number of vertices of $\theta$ that have label $k$. We then define $\mathscr{S}$ as the set of all trees in $\overline{\mathbf{T}}$ that have at most one spine, and whose labels take each integer value only finitely many times:
\begin{equation*}
\mathscr{S} = \mathbf{T} \cup \left\{ \theta \in \mathbf{T}_{\infty}: \, \forall l \geq 1, \, N_l (\theta) < \infty \text{ and $\theta$ has a unique spine}\right\}.
\end{equation*}
A tree $\theta \in \mathscr{S}$ can  be coded by two pairs of contour functions, $( C_{\theta}^{(L)},V_{\theta}^{(L)}): \mathbb{R}_+ \to \mathbb{R}_+ \times \mathbb{R}_+$ and $( C_{\theta}^{(R)},V_{\theta}^{(R)}) : \mathbb{R}_+ \to \mathbb{R}_+ \times \mathbb{R}_+$, each pair coding one side of the spine. 
Note that to define the pair $( C_{\theta}^{(L)},V_{\theta}^{(L)})$, we follow the contour of the tree from the left to the right as before, but 
in order to define $( C_{\theta}^{(R)},V_{\theta}^{(R)})$ we follow the contour from the right to the left.
The definition of these contour functions should be clear from Fig. \ref{fig:contourspine}. Note that the functions $C_{\theta}^{(L)}$, $V_{\theta}^{(L)}$, $C_{\theta}^{(R)}$ and $V_{\theta}^{(R)}$ tend to infinity at infinity.

\begin{figure}[!ht]
\begin{center}
\includegraphics[width=0.75\textwidth]{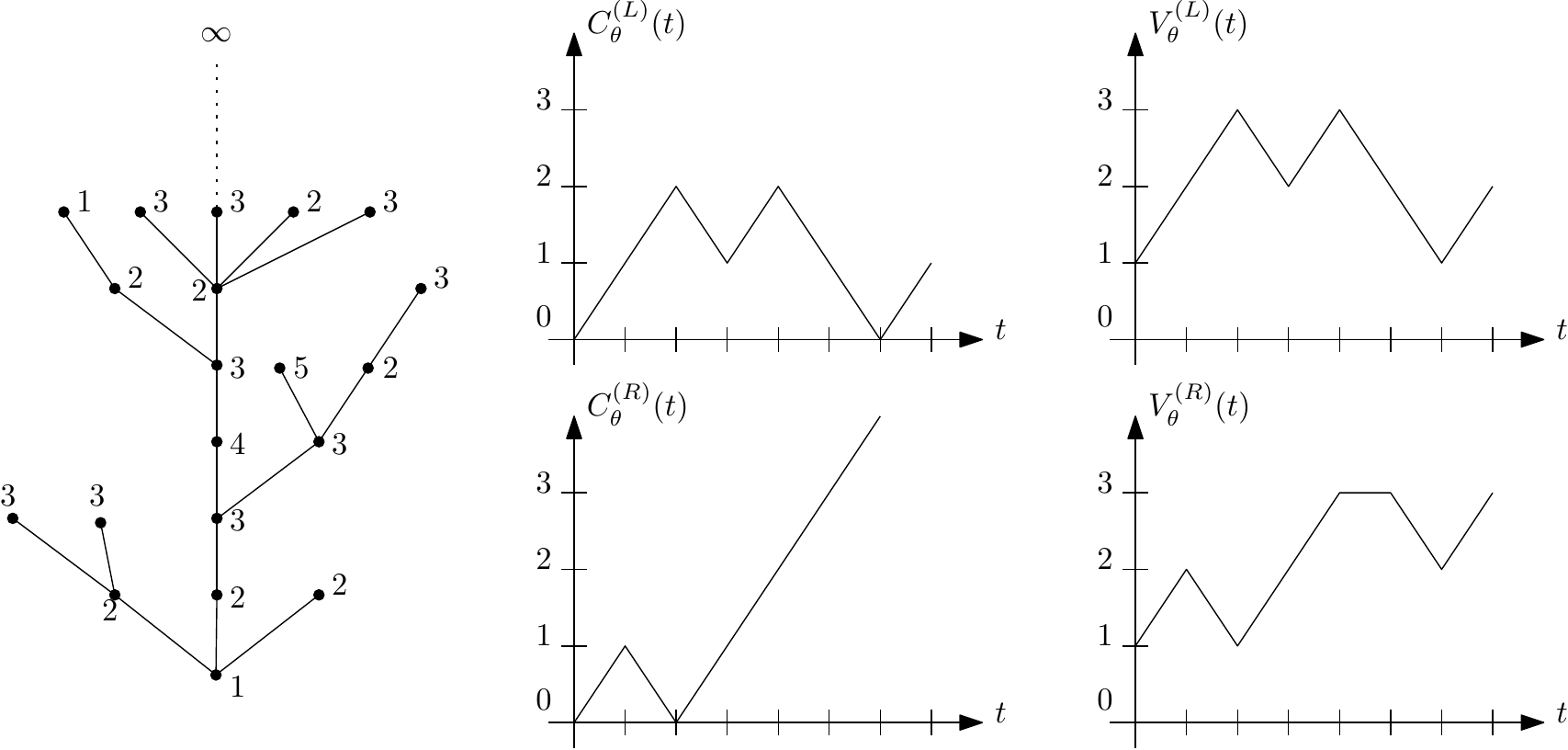}
\end{center}
\caption{An infinite well-labeled tree $\theta$ and its contour functions $( C_{\theta}^{(L)},V_{\theta}^{(L)})$, $( C_{\theta}^{(R)},V_{\theta}^{(R)})$.}
\label{fig:contourspine}
\end{figure}

\subsubsection{Planar maps and quadrangulations}
\label{subsubsec:quadrangulations}

A {\it planar map} is a proper embedding of a finite
connected graph in the two-dimensional sphere $\mathbb{S}^2$. Loops and multiple edges
are a priori allowed. The faces of the map are the connected components of the complement
of the union of edges. A planar map is rooted 
if it has a distinguished oriented edge called the root edge,
whose  origin is called the root vertex. In what follows, planar maps
are always rooted, even if this is not explicitly specified.
Two rooted planar maps
are said to be equivalent if the second one is the image of the first one
under an orientation-preserving
homeomorphism of the sphere, which also preserves the root edges.
Two equivalent planar maps will always be identified.

The vertex set of a planar map
will be equipped with the graph distance $d_{gr}$: if $v$ and $v'$ are two vertices, $d_{gr}(v,v')$ is the minimal number of edges on a path from $v$ to $v'$.

A planar map is a \emph{quadrangulation} if all its faces have degree $4$, that is $4$ adjacent edges
(one should count edge sides, so that if an edge lies entirely inside a face
it is counted twice).

Let us introduce infinite quadrangulations using Krikun's approach in \cite{Kr}.
For every integer $n \geq 1$, we denote the set of all rooted quadrangulations with $n$ faces by $\mathbf{Q}_n$, and we set
$$\mathbf{Q} = \bigcup_{n \geq 1} \mathbf{Q}_n.$$
For every $q,q' \in \mathbf{Q}$, we define
\[
D \left( q,q' \right) =  \left( 1 + \sup \left\{ r: \, M_{r}(q) = M_{r}(q') \right\} \right)^{-1}
\]
where, for $r \geq 1$, $M_{r}(q)$ is the rooted planar map obtained by keeping only those edges of $q$ that are adjacent to a face having at least one vertex at distance strictly smaller than $r$ from the root. By convention, $\sup \emptyset = 0$. Note that $M_{r}(q)$ is not a quadrangulation 
in general (it should be viewed as a quadrangulation with a boundary) but is still a planar map. Then $(\mathbf{Q},D)$ is a metric space. Denote by $(\overline{\mathbf Q},D)$ the completion of this space. We call (rooted) \emph{infinite quadrangulations} the elements of $\overline{\mathbf Q}$ that are not finite quadrangulations and we denote the set of all such quadrangulations by $\mathbf{Q}_{\infty}$.

Note that one can extend the function $q \in \mathbf{Q} \mapsto M_{r}(q)$ to a continuous function on $\overline{\mathbf Q}$. 
Suppose that $q\in\mathbf{Q}_{\infty}$. When $r$ varies,
the planar maps $M_r(q)$ are consistent in the sense that if $r<r'$ the
planar map $M_{r}(q)$ is naturally interpreted as the union of the faces of $M_{r'}(q)$ that have a vertex at distance strictly smaller than $r$ from the root.
Thanks to this observation, we can make sense of the vertex set of $q$ and of the graph distance on this vertex set.

The vertex set of a (finite or infinite) quadrangulation $q$ will always be denoted by $V(q)$, and the root vertex of $q$
will be denoted by $\partial$. 

\subsection{Schaeffer's correspondence}
\label{subsec:schaffer}

The relations between quadrangulations and labeled trees come from the following key result \cite{CV,Schaeffer}.
There exists a bijection $\Phi_n$, called Schaeffer's bijection, from $\mathbf{T}_n$ onto $\mathbf{Q}_n$ that enjoys the following property: if $\theta = (\tau, (\ell(v))_{v\in\tau}) \in \mathbf{T}_n$, then, for every integer $k \geq 1$ one has
\[
\left| \left\{ a \in V(\Phi_n(\theta)) : \, d_{gr}(\partial,a) = k \right\} \right| = \left| \left\{ v \in \tau : \, \ell(v) = k \right\} \right|.
\]

Schaeffer's bijection has been extended to the infinite setting in \cite{CD}: There exists
a one-to-one mapping $\Phi$ from $\mathscr{S}$ into $\overline{\mathbf Q}$ such that, for every $\theta = (\tau, (\ell(v))_{v\in\tau}) \in \mathscr{S}$, for every integer $k \geq 1$ one has
\[
\left| \left\{ a \in V(\Phi(\theta)) : \, d_{gr}(\partial,a) = k \right\} \right| = \left| \left\{ v \in \tau : \, \ell(v) = k \right\} \right|.
\]
Note however that $\Phi$ is not a bijection. There are infinite quadrangulations (in Krikun's sense) that cannot be written in the form $\Phi(\theta)$.

Let us describe the mapping $\Phi$ (see \cite{CD}, Section 6.2. for details). Fix a tree $\theta=(\tau,\ell) \in \mathscr{S}$ and 
assume that $\tau$ is infinite (the case when $\tau$ is finite is similar and easier to describe). Consider an embedding of $\tau$ in the sphere $\mathbb{S}^2$, such that every sequence $p = (p_n)_{n\in \mathbb{N}}$ of points of $\mathbb{S}^2$ belonging to distinct edges of $\tau$, has a unique accumulation point $\triangle \in \mathbb{S}^2$. Recall that a {\emph corner} of $\tau$ is a sector between two consecutive edges around a vertex. The label of the corner is the label of the corresponding vertex.

We first add a vertex $\partial$ in the complement of $\tau \cup \{\triangle\}$. Then, for every vertex $v$ of $\tau$ and every corner $c$ of $v$, an edge is added according to the following rules:
\begin{itemize}
 \item If $\ell(v) = 1$, we draw an edge between the corner $c$ and $\partial$ (see Fig. \ref{fig:Schaeffer}, left).
 \item If $c$ is on the right side of the spine, if $\ell(v) \geq 2$, and if there exists a corner with label
  $\ell(v) - 1$ that is visited after $c$ in the contour of the right side of the spine, we
  draw an edge between $c$ and the first such corner (see Fig. \ref{fig:Schaeffer}, left).
 \item If $c$ is on the right side of the spine, if $\ell(v) \geq 2$, and if there is no corner with label
  $\ell(v) - 1$ that is visited after $c$ in the contour of the right side of the spine, we draw an edge between
  $c$ and the corner on the left side of the spine with label $\ell(v) -1$ that is the last one to be visited during
  the contour of the left side of the spine (see Fig. \ref{fig:Schaeffer}, middle).
 \item If $c$ is on the left side of the spine and if $\ell(v) \geq 2$, we draw an edge between $c$ and the
  corner with label $\ell(v) - 1$ that is the last one to be visited before $c$ during the contour of the left side of the spine (see Fig. \ref{fig:Schaeffer}, right).
\end{itemize}
The construction can be made in such a way that edges do not intersect. The resulting (infinite) embedded planar graph whose vertices are the vertices of $\tau$ and the extra vertex $\partial$, and whose edges are obtained by the preceding prescriptions, 
is rooted at the oriented edge between $\partial$ and the first corner of $\emptyset$. This
embedded random graph $\Phi(\theta)$ can be interpreted as an infinite quadrangulation
in Krikun's sense. Moreover, for each vertex $v$ of $\tau$, the distance $d_{gr}(\partial,v)$
between the root vertex $\partial$ and $v$ in the map $\Phi(\theta)$ coincides with the label $\ell(v)$. 
\begin{figure}[!t]
\begin{center}
\includegraphics[width=0.75\textwidth]{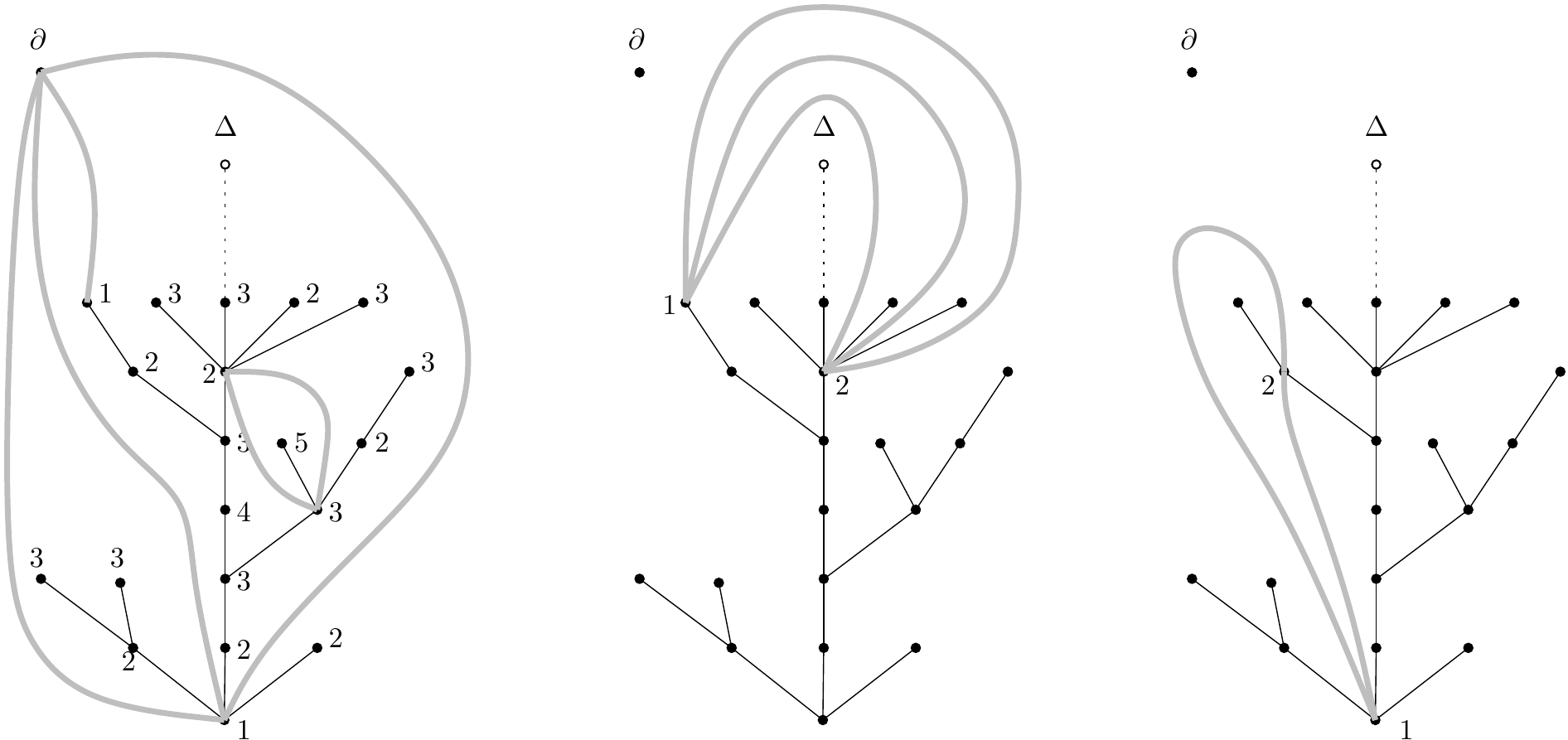}
\caption{Construction of a few edges in Schaeffer's correspondence.}
\label{fig:Schaeffer}
\end{center}
\end{figure}

\subsection{The uniform infinite quadrangulation}
\label{sec:uniform infinite}

In this section, we collect the known results about the uniform infinite quadrangulation and
the uniform infinite well-labeled tree.

\begin{theorem}[\cite{Kr}]
\label{krikun}
For every $n \geq 1$ let $\nu_n$ be the uniform probability measure on
$\mathbf{Q}_n$. The sequence $(\nu_n)_{n \in \mathbb{N}}$ converges to a probability measure $\nu$, in the sense of weak convergence  of probability measures on $(\overline{\mathbf Q}, D)$. Moreover, $\nu$ is supported on the set of infinite quadrangulations. A random quadrangulation distributed according to $\nu$ will be called a uniform infinite quadrangulation.
\end{theorem}
This probability measure is connected with the law of the uniform infinite well-labeled tree, which appears in the next theorem. Recall that 
$d$ stands for the distance on the space of labeled trees.

\begin{theorem}[\cite{CD}]
For every $n \geq 1$, let $\mu_n$ be the uniform probability measure on the set of all well-labeled trees with $n$ edges.
The sequence $(\mu_n)_{n \in \mathbb{N}}$ converges weakly to a probability measure $\mu$ in the sense of weak convergence of probability measures on $( \overline{\mathbf{T}}, d)$. Moreover, $\mu$ is supported on the set $\mathscr{S} \subset \mathbf{T}_{\infty}$. A random tree distributed according to $\mu$ will be called a uniform infinite well-labeled tree.
\end{theorem}

It was proved in previous work \cite{Me} that $\nu$ is the image of $\mu$ under the mapping $\Phi$ (the extended Schaeffer's correspondence) described in subsect. \ref{subsec:schaffer}. This is stated in the next theorem.
\begin{theorem}[\cite{Me}]
\label{thme}
For every Borel subset $A$ of $\overline{\mathbf Q}$ one has
\[
\nu(A) = \mu \left( \Phi^{-1} (A) \right).
\]
\end{theorem}
Informally, we may say that the uniform infinite quadrangulation is coded by the uniform infinite well-labeled tree. 

For our purposes, we do not really need
the preceding results. We will mainly use the description of the probability measure $\mu$ in Theorem \ref{th:descmu} below, and the fact that the
uniform infinite quadrangulation is obtained from a tree distributed according to $\mu$ via Schaeffer's correspondence.

In order to give a precise description of the measure $\mu$, we need a few more definitions.
Let $\theta=(\tau,(\ell(v))_{v\in\tau})$ be an infinite tree in $\mathscr{S}$ 	and let $n\geq 0$. If $v_n$ is the (unique) vertex at generation $n$ in the spine of $\theta$, we denote the label of $v_n$ by $X_n(\theta)=\ell(v_n)$. The (labeled) trees attached to $v_n$ respectively on the left side and on the right side of the spine are denoted by $L_n(\theta)$ and $R_n(\theta)$. More precisely, $L_n(\theta)=(\tau_{L_n},(\ell_{L_n}(v))_{v\in \tau_{L_n}})$, where 
$\tau_{L_n}=\{v\in{\mathcal U}: v_nv\in\tau\hbox{ and }v_nv\prec v_{n+1}\}$, and $\ell_{L_n}(v)=\ell(v_nv)$ for every $v\in\tau_{L_n}$, and a similar
definition holds for $R_n(\theta)$.

For every integer $l\in\mathbb{Z}$ we denote by $\rho_l$ the law of the Galton-Watson tree with geometric offspring distribution with parameter $1/2$
(see e.g. \cite{randomtrees}), labeled according to the following rules. The root has label $l$ and every other vertex has a label chosen uniformly in $\{m-1,m,m+1\}$ where $m$ is the label of its parent, these choices being made independently for every vertex. Then, $\rho_l$
is a probability measure on the space of all labeled trees. Moreover, for every labeled tree 
$\theta$ with $n$ edges and root label $l$, $\rho_l(\theta) = \frac{1}{2} 12^{-|\theta|}$. Since the cardinality of the set of all plane trees with $n$ edges is the Catalan number of order $n$, we easily get
\begin{align}
\label{lifetime1}
\rho_{l} \left( |\theta |= n  \right) & = \rho_{0} \left( |\theta |= n  \right) = \frac{n^{-3/2}}{2\sqrt{\pi}} + {\it O} \left( n^{-5/2} \right) \\
\label{lifetime2}
\rho_{l} \left( |\theta | \geq n  \right) & = \rho_{0} \left( |\theta | \geq n  \right) =  {\it O} \left( n^{-1/2} \right)
\end{align}
as $n$ goes to infinity.

Denote by $V_*=V_*(\theta)$ the minimal label in $\theta$. Suppose now that $l\geq 1$.
Proposition 2.4 of \cite{CD} shows that
\begin{equation}
\label{minimallabel}
\rho_l(V_*>0)= \frac{l (l+3)}{(l+1)(l+2)}.
\end{equation}
We define another probability measure $\widehat\rho_l$ on labeled trees by setting
$$\widehat\rho_l=\rho_l(\cdot\mid V_*>0).$$

We will very often use the bound $\widehat{\rho}_l \leq 2 \rho_l$, which holds for every $l\geq 1$ from the explicit formula 
for $\rho_l(V_*>0)$.

\smallskip

\begin{theorem}[\cite{CD}]
\label{th:descmu}
Let $\Theta$ be a random labeled tree distributed according to $\mu$. Write $X_n = X_n(\Theta)$ for every $n \geq 0$.
\begin{enumerate}
\item The process $X=(X_n)_{n \geq 0}$ is a Markov chain with transition kernel $\Pi$ defined by
\begin{align*}
\Pi (l,l-1) & = \frac{(w_l)^2}{12 d_l} d_{l-1} & \text{if $l \geq 2$,}\\
\Pi (l,l) & = \frac{(w_l)^2}{12} & \text{if $l \geq 1$,}\\
\Pi (l,l+1) & = \frac{(w_l)^2}{12 d_l} d_{l+1} & \text{if $l \geq 1$},
\end{align*}
where
\begin{align*}
w_{l} & = 2 \frac{l (l +3)}{(l +1)(l +2)},\\
d_{l} & = \frac{2w_{l}}{560} (4 l^4 + 30 l^3 + 59 l^2 + 42 l + 4).
\end{align*}
\item Conditionally given $(X_n)_{n \geq 0} = (x_n)_{n
\geq 0}$, the sequence $(L_n)_{n \geq 0}$ of subtrees of $\Theta$ attached to the left side of the spine and the sequence $(R_n)_{n \geq 0}$ of subtrees attached to the right side of the spine form two independent sequences of independent labeled trees distributed respectively according to the measures $\widehat{\rho}_{x_n}$, $n\geq 0$.
\end{enumerate}
\end{theorem}

We will also use the following proposition, 
which is proved in \cite{Me}. We keep the notation $(X_n)_{n\geq 0}$
for the labels on the spine of the tree $\Theta$.

\begin{proposition}[\cite{Me}]
\label{Bes9}
The sequence of processes $\left( \sqrt{\frac{3}{2n}} X_{\lfloor nt \rfloor} \right)_{t \geq 0}$ converges in distribution in the Skorokhod sense to a nine-dimensional Bessel process started at $0$.
\end{proposition}

We refer to Chapter XI of \cite{RY} for extensive information about Bessel processes.

\subsection{The Brownian snake}
\label{subsec:browniansnake}

In this section we collect some facts about the Brownian snake that we will use later. We refer to \cite{LGsnake} for a more complete presentation of the Brownian snake.

The Brownian snake is a Markov process taking values in the space $\mathcal{W}$ of all finite real paths. An element of $\mathcal{W}$ is simply a continuous mapping $\mathrm{w}: [0,\zeta] \to \mathbb{R}$, where $\zeta = \zeta_{(\mathrm{w})} \geq 0$ depends on $\mathrm{w}$ and is called the lifetime of $\mathrm{w}$. The endpoint (or tip) of $\mathrm{w}$ will be denoted by $\widehat{\mathrm{w}} = \mathrm{w}(\zeta)$. The range of $\mathrm{w}$ is denoted by $\mathrm{w}[0,\zeta_{(\mathrm{w})}]$. If $x \in \mathbb{R}$, we denote the subset of paths with initial point $x$ by $\mathcal{W}_x$. The trivial path in $\mathcal{W}_x$ such that $\zeta_{(\mathrm{w})} = 0$ is identified with the point $x$. The set $\mathcal{W}$ is a Polish space for the distance
\[
d_{\mathcal{W}}(\mathrm{w},\mathrm{w}') = |\zeta_{(\mathrm{w})} - \zeta_{(\mathrm{w}')}| + \sup_{t \geq 0} |\mathrm{w}(t \wedge \zeta_{(\mathrm{w})}) - \mathrm{w}'(t \wedge \zeta_{(\mathrm{w'})})|.
\]

The canonical space $\Omega = C(\mathbb{R}_+, \mathcal{W})$ is equipped with the topology of uniform convergence on every compact subset of $\mathbb{R}_+$. The canonical process on $\Omega$ is denoted by $W_s(\omega) = \omega(s)$ for $\omega \in \Omega$ and we write $\zeta_s = \zeta_{(W_s)}$ for the lifetime of $W_s$.

Let $\mathrm{w} \in \mathcal{W}$. The law of the 
(one-dimensional) Brownian snake started from $\mathrm{w}$ is the probability $\mathbb{P}_{\mathrm{w}}$ on $\Omega$ which can be characterized as follows. First, the process $(\zeta_s)_{s \geq 0}$ is under $\mathbb{P}_{\mathrm{w}}$ a reflected Brownian motion in $[0, \infty[$ started from $\zeta_{(\mathrm{w})}$. Secondly, the conditional distribution of $(W_s)_{s \geq 0}$ knowing $(\zeta_s)_{s \geq 0}$, which is denoted by $Q_{\mathrm{w}}^{\zeta}$, is characterized by the following properties:
\begin{enumerate}
\item $W_0 = \mathrm{w}$, $Q_{\mathrm{w}}^{\zeta}$ a.s.
\item The process $(W_s)_{s \geq 0}$ is time-inhomogeneous Markov under $Q_{\mathrm{w}}^{\zeta}$. Moreover, if $0 \leq s \leq s'$,
\begin{itemize}
\item $W_{s'}(t) = W_s(t)$ for every $t \leq m(s,s') = \inf_{[s,s']} \zeta_r$, $\Theta_{\mathrm{w}}^{\zeta}$ a.s.
\item $\left( W_{s'}(m(s',s) + t) - W_{s'}(m(s,s'))\right)_{0 \leq t \leq \zeta_{s'} - m(s,s')}$ is independent of $W_s$ and distributed under $Q_{\mathrm{w}}^{\zeta}$ as a Brownian motion started at $0$.
\end{itemize}
\end{enumerate}
Informally, the value $W_s$ of the Brownian snake at time $s$ is a random path with a random lifetime $\zeta_s$ evolving like a reflected Brownian motion in $[0, \infty[$. When $\zeta_s$ decreases, the path is erased from its tip, and when $\zeta_s$ increases, the path is extended by adding ``little pieces'' of Brownian paths at its tip.

\smallskip

We denote the It\^o measure of positive excursions  by ${\mathbf n}(\mathrm{d}e)$ (see e.g. Chapter XII of \cite{RY}). This is a $\sigma$-finite measure on the space $C(\mathbb{R}_+,\mathbb{R}_+)$. We write
\[
\sigma(e) = \inf \{ s>0 : \, e(s) = 0\}
\]
for the duration of an excursion $e$. For $s>0$, ${\mathbf n}_{(s)}$ denotes the conditioned probability measure $\mathbf{n}( \cdot \, | \, \sigma = s)$. Our normalization of the It\^o measure is fixed by the relation
\begin{equation}
\label{normexcursion}
{\mathbf n} = \int_0^{\infty} \frac{\mathrm{d}s}{2\sqrt{2 \pi s^3}} {\mathbf n}_{(s)}.
\end{equation}

If $x \in \mathbb{R}$, the excursion measure $\mathbb{N}_x$ of the Brownian snake started at $x$ is defined by
\[
\mathbb{N}_x = \int_{C(\mathbb{R}_+, \mathbb{R}_+)} {\mathbf n}(\mathrm{d}e) Q_x^{e}.
\]
With a slight abuse of notation we will also write $\sigma(\omega) = \inf \{ s>0 : \, \zeta_s (\omega) = 0\}$ for $\omega \in \Omega$. We can then consider the conditioned measures
\[
\mathbb{N}_x^{(s)} = \mathbb{N}_x ( \cdot \, | \, \sigma = s) = \int_{C(\mathbb{R}_+, \mathbb{R}_+)} {\mathbf n}_{(s)}(\mathrm{d}e) Q_x^{e}.
\]

The range $\mathcal{R} = \mathcal{R}(\omega)$ is defined by $\mathcal{R} = \{ \widehat{W}_s : \, s \geq 0\}$. 
We have, for every $x>0$,
\begin{equation}
\label{range-estimate}
\mathbb{N}_x \left( \mathcal{R} \cap  ]- \infty , 0 ] \neq \emptyset \right) = \frac{3}{2x^2}.
\end{equation}
See e.g. Section VI.1 of \cite{LGsnake} for a proof .

\subsection{Convergence towards the Brownian snake}

In this section, we recall a standard result of convergence towards the Brownian snake.
Let ${\mathcal F}=(\theta_1,\theta_2,\ldots)$ be a sequence of independent labeled trees 
distributed according to the probability measure $\rho_0$. We denote by $C^{\mathcal F}=(C^{\mathcal F}(t))_{t\geq 0}$ the contour function of the forest ${\cal F}$, which is obtained by concatenating the
contour functions of the trees $\theta_1,\theta_2,\ldots$. Similarly, $V^{\mathcal F}=(V^{\mathcal F}(t))_{t\geq 0}$ is obtained by concatenating the spatial
contour functions of the trees $\theta_1,\theta_2,\ldots$. Note that this concatenation creates
no problem because the labels of the roots of $\theta_1,\theta_2,\ldots$ are all equal to $0$.

In the next statement, $(W_t)_{t\geq 0}$ is the Brownian snake under the probability measure 
${\mathbb P}_0$ and $(\zeta_t)_{t\geq 0}$ is the associated lifetime process.

\begin{proposition}
\label{convtoBS}
The sequence of processes
$$\Big(\frac{1}{n} C^{\mathcal F}(n^2t), 
\sqrt{\frac{3}{2n}}\;V^{\mathcal F}(n^2t)\Big)_{t\geq 0}$$
converge in distribution to the process
$(\zeta_{t},\widehat W_{t})_{t\geq 0}$
in the sense of weak convergence of the laws on the space $C({\mathbb R}_+,{\mathbb R}^2)$.
\end{proposition}

The convergence of contour functions in the proposition follows from the more general
Theorem 1.17 of \cite{randomtrees} (in our particular case, it is just a straightforward application
of Donsker's theorem). The joint convergence with the
spatial contour process can then be obtained as an easy application of the techniques
in \cite{JM}.

Theorem \ref{snakeinfinite} below provides an analogue of Proposition \ref{convtoBS}
when the forest of independent trees ${\mathcal F}$ is replaced by the forest of
subtrees branching from the left (or right) side of the spine of the
uniform infinite well-labeled tree. This replacement makes the proof much more involved,
essentially because of the positivity constraint on labels.

\section{Scaling limit of the uniform infinite well-labeled tree}
\label{sec:scaletree}

\subsection{The eternal conditioned Brownian snake}
\label{subsec:infinitesnake}

We start by introducing the {\it eternal conditioned Brownian snake},
which will appear in our limit theorem for the uniform infinite well-labeled tree.
Let $Z = (Z_t)_{t \geq 0}$ be a nine-dimensional Bessel process started at $0$. Conditionally given $Z$, let
\begin{equation*}
\mathcal{P} = \sum_{i \in I} \delta_{(r_i,\omega_i)}
\end{equation*}
be a Poisson point process on $\mathbb{R}_+ \times \Omega$ with intensity
\begin{equation}
\label{eq:pointsnakeintens}
2\;\mathbf{1}_{\{\mathcal{R}(\omega) \subset ]- Z_t, \infty[\}} \mathrm{d}t \,\mathbb{N}_0(\mathrm{d}\omega)
\end{equation}
where we recall that $\mathcal{R}(\omega)$ denotes the range of the snake.
We then construct our conditioned snake $W^{\infty}$ as a measurable function $G(Z, \mathcal{P})$ of the pair $(Z, \mathcal{P})$. Let us describe this function $G$. To simplify notation, we put
\[
\sigma_i = \sigma(\omega_i), \quad \zeta_s^i = \zeta_s(\omega_i), \quad
W_s^i = W_s(\omega_i)
\]
for every $i \in I$ and $s \geq 0$. For every $u \geq 0$, we set
\[
\tau_u = \sum_{i \in I} \mathbf{1}_{\{ r_i \leq u \}} \sigma_i.
\]
Then, if $s \geq 0$, there is a unique $u$ such that $\tau_{u-} \leq s \leq \tau_u$, and:
\begin{itemize}
\item Either there is a (unique) $i \in I$ such that $u = r_i$ and we set
\begin{align*}
& \zeta_s^{\infty} = u + \zeta_{s - \tau_{u-}}^i,\\
& W_s^{\infty}(t) = 
\begin{cases} Z_t & \text{if $t \leq u$,}\\
Z_u + W_{s - \tau_{u-}}^i(t-u) & \text{if $u<t \leq \zeta_s^{\infty}$}.
\end{cases}
\end{align*}
\item Or there is no such $i$, then $\tau_{s-}=u=\tau_s$ and we set
\begin{align*}
& \zeta_s^{\infty} = u, \\
& W_s^{\infty}(t) = Z_t, \quad t \leq u.
\end{align*}
\end{itemize}
These prescriptions define a continuous process $W^{\infty} = G(Z, \mathcal{P})$ with values in $\mathcal{W}$. As usual the head of $W^{\infty}$ at time $s$ is $\widehat{W}^{\infty}_s = W^{\infty}_s(\zeta_s^{\infty})$. We say that $W^\infty$ is an eternal conditioned Brownian snake.

The preceding construction can be reinterpreted by saying that the pair
$(\zeta^{\infty}_s,\widehat W^{\infty}_s)_{s\geq 0}$ is obtained by 
concatenating (in the appropriate order
given by the values of $r_i$) the functions
$$\left(r_i+\zeta^i_{s}, Z_{r_i}+ \widehat W^i_s\right)_{0\leq s\leq \sigma_i}.$$
In particular, it is easy to verify that, a.s. for every $u\geq 0$,
$$\tau_u=\sup\{s\geq 0: \zeta^\infty_s\leq u\}.$$
This simple observation will be useful later. 

If $K>0$ is fixed, an application of 
\eqref{range-estimate} gives for every $u> 0$,
$$P\Big[\inf_{s\geq \tau_u} \widehat W^\infty_s > K\Big]= E\Big[\exp- 3\int_u^\infty \Big( Z_s-K)^{-2} - (Z_s)^{-2}\Big) \mathrm{d}s\Big],$$
with the convention that the integral in the exponential is infinite if $Z_s\leq K$ for some $s\geq u$. 
The right-hand side of the previous display tends to $1$ as $u\to\infty$, and
it follows that
\begin{equation}
\label{transieternal}
\lim_{s\to\infty} \widehat W^\infty_s = +\infty\;,\quad{\rm a.s.}
\end{equation}

Suppose that conditionally given $Z$, $\widetilde{\mathcal P}$ is another Poisson measure with the same intensity as $\mathcal{P}$, and that $\mathcal{P}$ and $\widetilde{\mathcal P}$ are independent conditionally given $Z$. Then let $W^\infty=G(Z,\mathcal{P})$ as before and also set
$\widetilde W^\infty=G(Z,\widetilde{\mathcal P})$. We say that $(W^\infty,\widetilde W^\infty)$ is a {\it pair of correlated eternal conditioned Brownian snakes} (driven by the Bessel process $Z$).

\subsection{Convergence of the rescaled uniform infinite well-labeled tree}
\label{subsec:prooftree}

Throughout this subsection, we consider a uniform infinite well-labeled tree $\Theta$, and we use the
notation introduced in Theorem \ref{th:descmu}: In particular $X_n$, $n\in{\mathbb Z}_+$ are the labels
along the spine of $\Theta$, and $L_n$ and $R_n$, $n\in{\mathbb Z}_+$, are the subtrees attached
respectively to the left side and to the right side of the spine.
Recall that the left side (resp. right side) of the spine can be coded by the contour functions $( C^{(L)}, V^{(L)} )$ (resp. $( C^{(R)}, V^{(R)} )$). The main result of this section gives the joint convergence of these suitably rescaled random functions towards a pair of correlated eternal conditioned Brownian snakes.

\begin{theorem}
\label{snakeinfinite}
Let $( W^{(L)},W^{(R)})$ be a pair of correlated eternal conditioned Brownian snakes. We have the joint convergence in distribution:
\begin{align}
\label{eq:snakeinf}
&\Big( \Big( \frac{1}{n} C^{(L)}(n^2 s) , \sqrt{\frac{3}{2n}} V^{(L)}(n^2 s)\Big)_{s \geq 0},
\Big( \frac{1}{n} C^{(R)}(n^2 s) , \sqrt{\frac{3}{2n}} V^{(R)}(n^2 s) \Big)_{s \geq 0} \Big) \notag \\
& \qquad \qquad \underset{n\to\infty}{\overset{\mathrm{(d)}}{\longrightarrow}}
\Big( \Big(\zeta^{(L)}_s,\widehat{W}^{(L)}_s \Big)_{s \geq 0},
\Big(\zeta^{(R)}_s,\widehat{W}^{(R)}_s \Big)_{s \geq 0} \Big).
\end{align}
where $\zeta^{(L)}_s = \zeta_{( W^{(L)}_s)}$, resp. $\zeta^{(R)}_s = \zeta_{( W^{(R)}_s)}$, for every $s \geq 0$. The convergence in distribution
{\rm(\ref{eq:snakeinf})} holds in the sense of weak convergence of laws of processes in the space 
$C({\mathbb R}_+,{\mathbb R}^2)^2$.
\end{theorem}

Before proving Theorem \ref{snakeinfinite}, we will establish a few preliminary results.
For every finite labeled tree $\theta$ and every $t\geq 0$, we set
\begin{equation*}
\Big( C_{\theta}^{(n)}(t) , V_{\theta}^{(n)}(t) \Big) = \Big( \frac{1}{n}C_{\theta}(n^2 t), \sqrt{\frac{3}{2n}} V_{\theta}(n^2t) \Big),
\end{equation*}
where $\left( C_{\theta},V_{\theta} \right)$ is the pair of contour functions of $\theta$. In addition, we also write 
$$\mathcal{R} \left(V_{\theta}^{(n)} \right)=
\{V_{\theta}^{(n)}(t):t \geq 0\}.$$

\begin{proposition}
\label{snakeconvlemma}
Let $\varphi$ be a bounded continuous function from $C(\mathbb{R}_+,\mathbb{R})^2 \times \mathbb{R}_+$ into $\mathbb{R}_+$. Assume that there exists $\eta > 0$ such that $\varphi(f,g,s) = 0$ if $s \leq \eta $. Fix $z > 0$ and let $(x_n)_{n\in \mathbb{N}}$ be a sequence of positive integers such that $\sqrt{\frac{3}{2n}}\, x_n \rightarrow z$ as $n$ goes to $\infty$. We have the following convergence:
\[
n \,\widehat{\rho}_{x_n} \Big( \varphi \Big( C_{\theta}^{(n)},V_{\theta}^{(n)}, \frac{2 |\theta|}{n^2} \Big) \Big) \underset{n \to \infty}{\longrightarrow} 2\,\mathbb{N}_z \Big( \varphi ( \zeta,\widehat{W} , \sigma)
\mathbf{1}_{\{ \mathcal{R}  \subset ] 0 , \infty [ \} } \Big).
\]
\end{proposition}

\begin{proof} Recall the notation 
$$w_l=2\frac{l(l+3)}{(l+1)(l+2)}= 2\,\rho_l(V_*>0),$$
for every integer $l\geq 1$.
Fix $K>\eta$. Then, for every integer $n\geq 1$,
\begin{align}
\label{snakeconv}
n & \widehat{\rho}_{x_n} \Big( \varphi \Big(  C_{\theta}^{(n)},V_{\theta}^{(n)} ,\frac{2 |\theta|}{n^2} \Big) \Big) \notag\\
& = 2 n w_{x_n}^{-1} \rho_{x_n} \Big(  \mathbf{1}_{\{ \mathcal{R} ( V_{\theta}^{(n)} ) \subset ] 0 , \infty [ \}} \varphi \Big(  C_{\theta}^{(n)},V_{\theta}^{(n)}, \frac{2 |\theta|}{n^2} \Big) \Big) \notag \\
& = 2 n w_{x_n}^{-1} \sum_{k = \lfloor \eta n^2 /2 \rfloor}^{\lfloor K n^2 \rfloor}  \rho_{x_n} ( |\theta |= k )
\rho_{x_n} \Big( \mathbf{1}_{\{ \mathcal{R}( V_{\theta}^{(n)} ) \subset ] 0 , \infty [  \}} \varphi \Big(  C_{\theta}^{(n)},V_{\theta}^{(n)} , \frac{2 |\theta|}{n^2} \Big) \, \Big| \, |\theta| = k \Big) \notag\\
& \quad + 2 n w_{x_n}^{-1} \rho_{x_n} ( |\theta | > K n^2 )
\rho_{x_n} \Big( \mathbf{1}_{\{ \mathcal{R} ( V_{\theta}^{(n)} ) \subset ] 0 , \infty [  \}} \varphi \Big(  C_{\theta}^{(n)},V_{\theta}^{(n)} , \frac{2 |\theta|}{n^2} \Big) \, \Big| \, |\theta| > K n^2 \Big) .
\end{align}
The first term in the right-hand side of (\ref{snakeconv}) can be written as
\begin{equation}
\label{snakeconv2}
2 n^3 w_{x_n}^{-1} \int_{\frac{ \lfloor \eta n^2 /2 \rfloor}{n^2}}^{\frac{\lfloor Kn^2 \rfloor+1}{n^2}} \mathrm{d}s  \, \rho_{x_n}( |\theta |= \lfloor s n^2 \rfloor )
\rho_{x_n}\!\Big( \mathbf{1}_{\{ \mathcal{R} ( V_{\theta}^{(n)}) \subset ] 0 , \infty [  \}}
\varphi \Big(  C_{\theta}^{(n)},V_{\theta}^{(n)}, \frac{2 \lfloor s n^2 \rfloor }{n^2} \Big) \, \Big|  \, |\theta| = \lfloor s n^2 \rfloor \Big).
\end{equation}

In order to investigate the behavior of the quantity (\ref{snakeconv2}) as $n\to\infty$, we 
use a result about the convergence of discrete snakes.
Fix $y>0$ and let $(y_k)_{k\in \mathbb{N}}$ be a sequence of positive integers such that $(9/8k )^{1/4} y_k \rightarrow y$ as $n$ goes to $\infty$. Let $\left( \mathbf{W}_t\right)_{t \in [0,1]}$ be distributed according to $\mathbb{N}^{(1)}_{y}$ (see subsect. \ref{subsec:browniansnake}).
Then $(\mathbf{e}_t)_{t \in [0,1]} := (\zeta_{(\mathbf{W}_t)})_{t \in [0,1]}$ is a normalized Brownian excursion. Theorem 4 of \cite{CS} (see also Theorem 2 of \cite{JM}) implies that the law of the pair
\[
\Big( \frac{C_{\theta}(2kt)}{\sqrt{2k}}, \left(\frac{9}{8}\right)^{1/4}\frac{V_{\theta}(2kt)}{k^{1/4}} \Big)_{t \in [0,1]} 
\]
under $\rho_{y_k} \left( \cdot \, \middle| |\theta| = k \right)$ converges as $k$ goes to infinity to the law of $(\mathbf{e}_t, \widehat{\mathbf{W}}_t )_{t \in [0,1]}$ in the sense of weak convergence of probability measures on $C([0,1] , \mathbb{R}^2)$. If $s> 0$ is fixed, we can apply
the previous convergence to integers $k$ of the form $k=\lfloor sn^2\rfloor$, noting that $(9/8\lfloor sn^2\rfloor)^{1/4}x_n$ converges 
to $(2s)^{-1/4}z$ under our assumptions, and we get
\begin{align*}
& \rho_{x_n} \Big( \mathbf{1}_{\{ \mathcal{R} ( V_{\theta}^{(n)} ) \subset ] 0 , \infty [ \}} \varphi \Big( C_{\theta}^{(n)},V_{\theta}^{(n)}, \frac{2 \lfloor sn^2 \rfloor}{n^2} \Big) \, \Big|  \, |\theta| = \lfloor s n^2 \rfloor \Big)\\
& \qquad \qquad \underset{n \to \infty}{\longrightarrow}
\mathbb{N}_{(2s)^{-1/4}z}^{(1)} \left( \mathbf{1}_{\left\{ \mathcal{R} \subset \left] 0 , \infty \right[  \right\}} \varphi \left(\sqrt{2s} \zeta_{(./2s)} , (2s)^{1/4}\widehat W_{(./2s)}, 2s \right) \right).
\end{align*}
To justify the latter convergence, we also use the property
\[
\mathbb{N}_{(2s)^{-1/4}z}^{(1)} \left( \inf_{t\in \mathbb{R}_+}  \widehat W_t = 0 \right) = 0,
\]
which follows from the fact that the law of the infimum of a Brownian snake driven by a normalized Brownian excursion $\mathbf{e}$ has no atoms: see the beginning of the proof of Lemma 7.1 in \cite{LG06}.

A scaling argument then gives
$$\mathbb{N}_{(2s)^{-1/4}z}^{(1)} \left( \mathbf{1}_{\left\{ \mathcal{R} \subset \left] 0 , \infty \right[  \right\}} \varphi \left(\sqrt{2s} \zeta_{(./2s)} , (2s)^{1/4}\widehat W_{(./2s)}, 2s \right) \right)
=\mathbb{N}_{z}^{(2s)} \left( \mathbf{1}_{\left\{ \mathcal{R}\subset \left] 0 , \infty \right[  \right\}} \varphi \left( \zeta , \widehat{W} , 2s \right) \right)$$
and thus we have proved, for every fixed $s>0$,
\begin{equation}
\label{janmar}
\rho_{x_n} \Big( \mathbf{1}_{\{ \mathcal{R} ( V_{\theta}^{(n)} ) \subset ] 0 , \infty[ \}} \varphi \Big(  C_{\theta}^{(n)},V_{\theta}^{(n)}, \frac{2 \lfloor sn^2 \rfloor}{n^2} \Big) \, \Big|  \, |\theta| = \lfloor s n^2 \rfloor \Big) \underset{n \to \infty}{\longrightarrow}
\mathbb{N}_{z}^{(2s)} \left( \mathbf{1}_{\left\{ \mathcal{R}  \subset \left] 0 , \infty \right[  \right\}} \varphi \left( \zeta , \widehat{W} , 2s \right) \right).
\end{equation}

\smallskip

From the explicit formula for $w_l$, we have $w_l \geq 4/3$ for every $l > 0$. Using also \eqref{lifetime1}, we see that the following bound 
holds for all sufficiently large $n$: for every $s \in [\eta,K]$,
\begin{equation}
\label{domin}
2 n^3  w_{x_n}^{-1} \rho_{x_n} ( |\theta |= \lfloor s n^2 \rfloor )\;
\rho_{x_n}\!\Big(  \mathbf{1}_{\{ \mathcal{R} ( V_{\theta}^{(n)} ) \subset ] 0 , \infty [  \}}
\varphi \Big(  C_{\theta}^{(n)},V_{\theta}^{(n)}, \frac{2 |\theta|}{n^2} \Big) \, \Big|  \, |\theta| = \lfloor s n^2 \rfloor \Big)
\leq \frac{3}{2 \sqrt{\pi \eta^3}} \, \|\varphi\|_{\infty},
\end{equation}
where $\|\varphi\|_{\infty}$ is the supremum of $|\varphi|$.

We can use \eqref{lifetime1}, \eqref{janmar}, \eqref{domin}
(to justify dominated convergence) and the fact that $w_{x_n} \to 2$ as $n \to \infty$ to see that the 
quantity (\ref{snakeconv2}) converges as $n\to\infty$ to
\[
\int_{\eta}^K \frac{\mathrm{d}s}{2 \sqrt{\pi s^3}}
\mathbb{N}_{z}^{(2s)} \left( \mathbf{1}_{\left\{ \mathcal{R}  \subset \left] 0 , \infty \right[  \right\}} \varphi \left( \zeta ,\widehat{W}, 2s \right) \right)
= \int_0^K \frac{\mathrm{d}s}{2 \sqrt{\pi s^3}}
\mathbb{N}_{z}^{(2s)} \left( \mathbf{1}_{\left\{ \mathcal{R}  \subset \left] 0 , \infty \right[  \right\}} \varphi \left( \zeta ,\widehat{W}, 2s \right) \right).
\]
Since this holds for every $K > \eta$, we get by using \eqref{normexcursion} that
\[
\liminf_{n\to \infty}  n\, \widehat{\rho}_{x_n} \Big( \varphi \Big( C_{\theta}^{(n)},V_{\theta}^{(n)}, \frac{2 |\theta|}{n^2} \Big) \Big) \geq
2\, \mathbb{N}_{z} \left( \mathbf{1}_{\left\{ \mathcal{R}  \subset \left] 0 , \infty \right[  \right\}} \varphi \left( \zeta , \widehat{W} , \sigma \right) \right).
\]
Similar arguments, using also the estimate (\ref{lifetime2}), lead to
\[
\limsup_{n\to \infty}  n\, \widehat{\rho}_{x_n} \Big( \varphi \Big( C_{\theta}^{(n)},V_{\theta}^{(n)}, \frac{2 |\theta|}{n^2} \Big) \Big) \leq
\int_{\eta}^K \frac{\mathrm{d}s}{2 \sqrt{\pi s^3}}
\mathbb{N}_{z}^{(2s)} \left( \mathbf{1}_{\left\{ \mathcal{R}  \subset \left] 0 , \infty \right[  \right\}} \varphi \left( \zeta , \widehat{W} , 2 s \right) \right) + \frac{C}{\sqrt{K}} \, \|\varphi\|_{\infty}.
\]
with a constant $C$ that does not depend on $K$. By letting $K \to \infty$, we get
\[
\limsup_{n\to \infty}  n\, \widehat{\rho}_{x_n} \Big( \varphi \Big( C_{\theta}^{(n)},V_{\theta}^{(n)}, \frac{2 |\theta|}{n^2} \Big) \Big) \leq
2\, \mathbb{N}_z \left( \mathbf{1}_{\left\{ \mathcal{R}  \subset \left] 0 , \infty \right[ \right\}} \varphi \left( \zeta , \widehat{W} , \sigma \right) \right)
\]
which completes the proof.
\end{proof}

We now state a technical lemma, which will play an important role in the proof of Theorem \ref{snakeinfinite}.
We need to introduce some notation. For every integer $n\geq 1$ and every $h>0$, we set
\[
\tau^{(L,n,h)} = \frac{\lfloor n h  \rfloor}{n^2} + \sum_{i =0}^{\lfloor n h\rfloor} 2 n^{-2} |L_i|.
\]
This is the time needed in the rescaled contour of the left side of the spine to
explore the trees $L_i$, $0\leq i\leq \lfloor nh\rfloor$. Furthermore, for every
integer $k\geq 0$, we write $J_k$ for the unique index $i$ such that the 
vertex visited at time $k$ in the contour of the left side of the spine belongs to $L_i$. 

\begin{lemma}
\label{techinfinite}
Let $h>0$.
For every $\kappa>0$, we can find $\delta>0$ sufficiently small so that,
for all large integers $n$,
$$P\left[ \sup_{0\leq u<v\leq \tau^{(L,n,h)},\; v-u<\delta} \ \frac{1}{n}\left| J_{\lfloor n^2u\rfloor}
-J_{\lfloor n^2v\rfloor}\right| > \kappa \right] < \kappa.$$
\end{lemma}

\begin{proof}[Proof]
To simplify notation, we write $p_n(\kappa,\delta)$ for the probability that is
bounded in the lemma.
Suppose that there exist $u$ and $v$ with $0\leq u<v\leq \tau^{(L,n,h)}$ and $v-u<\delta$,
such that $| J_{\lfloor n^2u\rfloor}
-J_{\lfloor n^2v\rfloor}|>n\kappa$. Notice that all vertices belonging to the subtrees $L_i$
for indices $i$ such that $J_{\lfloor n^2u\rfloor}<i<J_{\lfloor n^2v\rfloor}$ are visited by
the contour of the left side of the spine between times $\lfloor n^2 u\rfloor$ and $\lfloor n^2 v\rfloor$. 
Hence
$$2\sum_{J_{\lfloor n^2u\rfloor}<i<J_{\lfloor n^2v\rfloor}} |L_i| \leq   \lfloor n^2 v\rfloor -\lfloor n^2 u\rfloor
\leq n^2 \delta +1.$$
Since $| J_{\lfloor n^2u\rfloor}
-J_{\lfloor n^2v\rfloor}|>n\kappa$,  we can find an integer $j$ of the form 
$j= l \lfloor n\kappa/2\rfloor$, with $1\leq l \leq nh/\lfloor n\kappa/2\rfloor$,  such that 
the inequalities $J_{\lfloor n^2u\rfloor}<i<J_{\lfloor n^2v\rfloor}$ hold for 
$i=j+1,j+2,\ldots,j+\lfloor n\kappa/2\rfloor$. 

It follows from the preceding considerations that
\begin{align*}
p_n(\kappa,\delta)
&\leq 
P\left[ \bigcup_{1\leq l \leq nh/\lfloor n\kappa/2\rfloor}
\left\{2 \sum_{i=1}^{\lfloor n\kappa/2\rfloor} 
|L_{l \lfloor n\kappa/2\rfloor + i}| \leq n^2\delta +1\right\}\right]\\
&\leq P\left[ \bigcup_{1\leq l \leq nh/\lfloor n\kappa/2\rfloor}
\left(\bigcap_{i=1}^{\lfloor n\kappa/2\rfloor} \left\{2 
|L_{l \lfloor n\kappa/2\rfloor + i}| \leq n^2\delta +1\right\}\right)\right].
\end{align*}

From Proposition \ref{Bes9}  and properties of the Bessel process, we can fix $\eta >0$ and $A>0$ such that
$$P\left[ \eta\sqrt{n} \leq X_i \leq A\sqrt{n},\,\forall i\in\{\lfloor n\kappa/2\rfloor,\ldots,\lfloor nh\rfloor
+\lfloor n\kappa/2\rfloor \}\right] >1-\kappa/2.$$
It follows that
\begin{align*}p_n(\kappa,\delta)
&\leq \frac{\kappa}{2}
+ \sum_{1\leq l \leq nh/\lfloor n\kappa/2\rfloor}P\left[ 
\bigcap_{i=1}^{\lfloor n\kappa/2\rfloor} \left\{2 
|L_{l \lfloor n\kappa/2\rfloor + i}| \leq n^2\delta +1, 
\eta\sqrt{n} \leq X_{l \lfloor n\kappa/2\rfloor + i} \leq A\sqrt{n}\right\}\right]\\
&\leq  \frac{\kappa}{2}
+ \frac{nh}{\lfloor n\kappa/2\rfloor}
\left( \sup_{ \eta\sqrt{n} \leq x\leq A\sqrt{n}} \widehat\rho_x\left(2|\theta| \leq n^2\delta +1\right)\right)
^{\lfloor n\kappa/2\rfloor}
\end{align*}
using the conditional distribution of the trees $L_i$ given the labels on the spine (Theorem 
\ref{th:descmu}). We can find a large constant $K>0$ such that,
for every sufficiently large $n$,
$$  \frac{\kappa}{2}
+ \frac{nh}{\lfloor n\kappa/2\rfloor} \left(1-\frac{K}{n}\right)^{\lfloor n\kappa/2\rfloor}
<\kappa.$$
To complete the proof of the lemma, we just have to observe that we can 
choose $\delta>0$ sufficiently small so that, for all $n$ large,
$$\inf_{ \eta\sqrt{n} \leq x\leq A\sqrt{n}} \widehat\rho_x\left(2|\theta| > n^2\delta +1\right)
\geq \frac{K}{n}.$$
This is indeed a consequence of Proposition \ref{snakeconvlemma},
together with the fact that 
$$\lim_{\delta\downarrow 0} {\mathbb N}_\eta \left(\sigma >\delta, {\mathcal R}\subset ]0,\infty[\right)
= {\mathbb N}_\eta \left({\mathcal R}\subset ]0,\infty[\right)=+\infty.$$
\end{proof}

We denote the rescaled contour functions of the labeled trees $L_i$ (resp. $R_i$) by $C_{L_i} ^{(n)}$ and $V_{L_i}^{(n)}$ (resp. $C_{R_i} ^{(n)}$
and $V_{R_i} ^{(n)}$), in agreement
with the notation introduced after Theorem \ref{snakeinfinite}. To simplify notation we also put
\[
X^{(n)}_t = \sqrt{\frac{3}{2n}}X_{\lfloor nt \rfloor}, \, t \geq 0.
\]

\begin{proposition}
\label{pointconv}
Fix $\varepsilon > 0$ and $h_0 > 0$. Let $\phi: \mathbb{D}(\mathbb{R}_+) \to \mathbb{R}$ and $\psi^{(L)}, \psi^{(R)} : \mathbb{R}_+ \times C(\mathbb{R}_+, \mathbb{R})^2 \times \mathbb{R}_+ \to \mathbb{R}_+$ be continuous functions. Assume that 
$\phi$ is bounded, and that $\psi^{(L)}$ and $\psi^{(R)}$ are Lipschitz with respect to the first variable and such that $\psi^{(L)}(h,f,g,s) = 0$ and $\psi^{(R)}(h,f,g,s) = 0$ if $h \geq h_0$ or $s \leq \varepsilon$. Then
\begin{align*}
E & \left[ \phi \left( X^{(n)} \right)
\exp \left( - \sum_{i = 0}^{\infty} \psi^{(L)} \left( \frac{i}{n} , C_{L_i}^{(n)} , V_{L_i}^{(n)} , \frac{2 |L_i|}{n^2} \right) \right)
\exp \left( - \sum_{i = 0}^{\infty} \psi^{(R)} \left( \frac{i}{n} , C_{R_i}^{(n)} , V_{R_i}^{(n)} , \frac{2 |R_i|}{n^2} \right) \right)
\right] \\
& \underset{n \to \infty}{\longrightarrow}
E\Bigg[ \phi \left( Z \right)
\exp \left( - 2 \int_0^{\infty} \mathrm{d}h \, \mathbb{N}_{Z_h} \left( \mathbf{1}_{ \left\{ \mathcal{R}  \subset ]0,\infty [ \right\} } \left( 1 - \exp - \psi^{(L)} \left(h, \zeta,\widehat{W} , \sigma  \right) \right) \right) \right) \\
& \qquad \qquad \qquad \qquad \times
\exp \left( - 2 \int_0^{\infty} \mathrm{d}h \, \mathbb{N}_{Z_h} \left( \mathbf{1}_{ \left\{ \mathcal{R}  \subset ]0,\infty [ \right\} } \left( 1 - \exp - \psi^{(R)} \left(h, \zeta ,\widehat{W} , \sigma \right) \right) \right) \right)
\Bigg],
\end{align*}
where $Z$ is a nine-dimensional Bessel process started from $0$.
\end{proposition}
\begin{remark}
We can interpret the limit in the theorem in terms of Poisson point processes. Conditionally given $Z$, let $(\mathcal{P}^{(L)},\mathcal{P}^{(R)})$ be a pair of independent Poisson point processes on $\mathbb{R}_+ \times \Omega$ with intensity given by \eqref{eq:pointsnakeintens}. Then, the exponential formula for Poisson point processes shows that the limit appearing in the proposition is equal to
\begin{align*}
E & \Bigg[ \phi \left( Z \right)
\exp \left( - \int \psi^{(L)} \left( h ,\zeta_{.}(\omega),
Z_h+\widehat W_{.}(\omega), \sigma(\omega) \right) \mathcal{P}^{(L)}( \mathrm{d}h, \mathrm{d}\omega ) \right) \\
& \qquad \qquad \qquad \qquad \times\exp \left( - \int \psi^{(R)} \left( h ,\zeta_{.}(\omega),
Z_h+\widehat W_{.}(\omega), \sigma(\omega) \right)  \mathcal{P}^{(R)}( \mathrm{d}h, \mathrm{d}\omega ) \right)
\Bigg] .
\end{align*}
\end{remark}

\begin{proof}
We have
\begin{align}
E &\left[ \phi \left( X^{(n)} \right)
\exp \left( - \sum_{i = 0}^{\infty} \psi^{(L)} \left( \frac{i}{n} , C_{L_i}^{(n)} , V_{L_i}^{(n)} , \frac{2 |L_i|}{n^2} \right) \right)
\exp \left( - \sum_{i = 0}^{\infty} \psi^{(R)} \left( \frac{i}{n} , C_{R_i}^{(n)} , V_{R_i}^{(n)} , \frac{2 |R_i|}{n^2} \right) \right)
\right] \notag \\
& = E \Bigg[ \phi \left( X^{(n)} \right)
\prod_{i=0}^{\infty} E \left[ \exp - \psi^{(L)} \left( \frac{i}{n} , C_{L_i}^{(n)} , V_{L_i}^{(n)} , \frac{2 |L_i|}{n^2} \right) \,\middle|\, X_i\right] \notag\\
& \qquad \qquad \qquad \qquad \times \prod_{i=0}^{\infty} E \left[ \exp - \psi^{(R)} \left( \frac{i}{n} , C_{L_i}^{(n)} , V_{L_i}^{(n)} , \frac{2 |R_i|}{n^2} \right) \,\middle|\, X_i\right] \Bigg] \label{ppcond}
\end{align}
using the independence of the subtrees $L_i$ and $R_i$ given the labels on the spine (Theorem \ref{th:descmu}).

Let us study the contribution of the left side of the spine in \eqref{ppcond}. By Theorem \ref{th:descmu} again,
\begin{align}
\prod_{i=0}^{\infty} & E \left[ \exp - \psi^{(L)} \left( \frac{i}{n} , C_{L_i}^{(n)} , V_{L_i}^{(n)} , \frac{2 |L_i|}{n^2} \right)\, \middle|\, X_i \right] \notag \\
& \qquad = \prod_{i=0}^{\infty} \widehat{\rho}_{X_i} \left( \exp - \psi^{(L)} \left( \frac{i}{n} , 
 C_{\theta}^{(n)},V_{\theta}^{(n)} , \frac{2 |\theta|}{n^2} \right) \right) \notag\\
& \qquad = \exp \sum_{i=0}^{\infty}  \log \widehat{\rho}_{X_i} \left( \exp - \psi^{(L)} \left( \frac{i}{n} ,  C_{\theta}^{(n)},V_{\theta}^{(n)} , \frac{2 |\theta|}{n^2} \right) \right) \notag\\
& \qquad = \exp n \int_{0}^{\infty} \mathrm{d}t \, \log \left( 1 - \widehat{\rho}_{X_{\lfloor nt \rfloor}} \left(1 - \exp - \psi^{(L)} \left( \frac{\lfloor n t \rfloor}{n} ,  C_{\theta}^{(n)},V_{\theta}^{(n)}, \frac{2 |\theta|}{n^2} \right) \right) \right) \label{leftproc}.
\end{align}

By Proposition \ref{Bes9} and the Skorokhod representation theorem we can find, for every $n \geq 1$, a process $\left( \widetilde{X}^n_k \right)_{k \geq 0}$ having the same distribution as $\left( X_k \right)_{k \geq 0}$, and a nine-dimensional Bessel process $Z$ started from $0$, such that almost surely, for every $a>0$, $ \left(\sqrt{\frac{3}{2n}} \widetilde{X}^n_{\lfloor nt \rfloor}\right)_{0 \leq t \leq a}$ converges uniformly to $(Z_t)_{0 \leq t \leq a}$ as $n$ goes to infinity. Using the Lipschitz property of $\psi^{(L)}$ in the first variable, together with the fact that $\psi^{(L)}(h,f,g,s) = 0$ if $s \leq \varepsilon$, we have, for some constant $K$,
\begin{align}
\label{tech28}
&\Big|n\, \widehat{\rho}_{\widetilde{X}_{\lfloor nt \rfloor}^{n}}\! \Big(1 - \exp - \psi^{(L)} \Big( \frac{\lfloor n t \rfloor}{n} ,  C_{\theta}^{(n)},V_{\theta}^{(n)} , \frac{2 |\theta|}{n^2} \Big)  \Big)
- n\, \widehat{\rho}_{\widetilde{X}_{\lfloor nt \rfloor}^{n}} \!
\Big(1 - \exp - \psi^{(L)} \Big( t ,  C_{\theta}^{(n)},V_{\theta}^{(n)} , \frac{2 |\theta|}{n^2} \Big)  \Big)\Big|\notag\\
\noalign{\smallskip}
&\quad \leq K \,\widehat{\rho}_{\widetilde{X}_{\lfloor nt \rfloor}^{n}} 
( |\theta| \geq \lfloor \varepsilon n^2 \rfloor /2 ) \leq 2K\,\rho_0( |\theta| \geq \lfloor\varepsilon n^2\rfloor/2),
\end{align}
which tends to $0$ as $n\to\infty$.
We then deduce from Proposition \ref{snakeconvlemma} 
that, for every fixed $t>0$,
\begin{align}
\label{eqmodif}
n\, \widehat{\rho}_{\widetilde{X}_{\lfloor nt \rfloor}^{n}} & \left(1 - \exp - \psi^{(L)} \left( t ,  C_{\theta}^{(n)},V_{\theta}^{(n)} , \frac{2 |\theta|}{n^2} \right)  \right) \notag\\
& \underset{n \to \infty}{\longrightarrow}
2 \mathbb{N}_{Z_t} \left( \mathbf{1}_{ \left\{ \mathcal{R}  \subset ]0,\infty [ \right\} } \left( 1 - \exp - \psi^{(L)} \left(t, \zeta , \widehat{W} , \sigma \right) \right) \right),\quad\hbox{a.s.}
\end{align}
From our assumptions on $\psi^{(L)}$, we have for every $t > 0$ and $n \geq 0$:
\begin{align*}
n \widehat{\rho}_{\widetilde{X}_{\lfloor nt \rfloor}^{n}} & \left(1 - \exp - \psi^{(L)} \left( \frac{\lfloor n t \rfloor}{n} ,  C_{\theta}^{(n)},V_{\theta}^{(n)} , \frac{2 |\theta|}{n^2} \right) \right)\\
& = n \widehat{\rho}_{\widetilde{X}_{\lfloor nt \rfloor}^{n}} \left( \mathbf{1}_{ \{t \leq h_0 +1 \}} \mathbf{1}_{ \{ |\theta| \geq \lfloor \varepsilon n^2 \rfloor /2 \} } \left( 1 - \exp - \psi^{(L)} \left(\frac{\lfloor n t \rfloor}{n} ,  C_{\theta}^{(n)},V_{\theta}^{(n)}, \frac{2 |\theta|}{n^2} \right) \right) \right) \\
& \leq \mathbf{1}_{ \{t \leq h_0 +1 \}} \, n \widehat{\rho}_{\widetilde{X}_{\lfloor nt \rfloor}^{n}} \left( |\theta| \geq \lfloor \varepsilon n^2 \rfloor /2 \right).
\end{align*}
It then follows from \eqref{lifetime2}
and the bound $\widehat \rho_l\leq 2\rho_l$ that there exists a constant $K' > 0$, which does not depend on $t$, such that for every $t > 0$ and every $n \geq 1$ one has:
\[
n \widehat{\rho}_{\widetilde{X}_{\lfloor nt \rfloor}^{n}} \left(1 - \exp - \psi^{(L)} \left( \frac{\lfloor n t \rfloor}{n} ,  C_{\theta}^{(n)},V_{\theta}^{(n)}, \frac{2 |\theta|}{n^2} \right) \right)
\leq K' \mathbf{1}_{ \{t \leq h_0 +1 \}}.
\]
Thus, we can use \eqref{tech28}, \eqref{eqmodif} and dominated convergence to see that the right-hand side of \eqref{leftproc}, with $X$ replaced by $\widetilde{X}^n$, converges a.s. to
\[
\exp - 2 \int_0^{\infty} \mathrm{d}t \, \mathbb{N}_{Z_t} 
\left( \mathbf{1}_{ \left\{ \mathcal{R} \subset ]0,\infty [ \right\} } \left( 1 - \exp - \psi^{(L)} \left(t, \zeta , \widehat{W} , \sigma \right) \right) \right)
\]
as $n \to \infty$. A similar analysis applies to the contribution of the right side of the spine in \eqref{ppcond}. Using the fact that  $\widetilde{X}^n$ has the same distribution as $X$
(so that the right-hand side of \eqref{ppcond} coincides with a similar expectation
involving $\widetilde{X}^n$)  we conclude that
\begin{align*}
E & \Bigg[ \phi \left( X^{(n)} \right) \exp \left( - \sum_{i = 0}^{\infty} \psi^{(L)} \left( \frac{i}{n} , C_{L_i}^{(n)} , V_{L_i}^{(n)} , \frac{2 |L_i|}{n^2} \right) \right)
\exp \left( - \sum_{i = 0}^{\infty} \psi^{(R)} \left( \frac{i}{n} , C_{R_i}^{(n)} , V_{R_i}^{(n)}, \frac{2 |R_i|}{n^2} \right) \right)
\Bigg] \\
& \underset{n \to \infty}{\longrightarrow} E \left[ \phi(Z) \exp - 2 \int_0^{\infty} \mathrm{d}t \, \mathbb{N}_{Z_t} \left( \mathbf{1}_{ \left\{ \mathcal{R} \subset ]0,\infty [ \right\} } \left( 1 - \exp - \psi^{(L)} \left(t, \zeta , \widehat{W} , \sigma \right) \right) \right) \right. \\
& \qquad \qquad \qquad \left. \times \exp - 2 \int_0^{\infty} \mathrm{d}t \, \mathbb{N}_{Z_t} \left( \mathbf{1}_{ \left\{ \mathcal{R} \subset ]0,\infty [ \right\} } \left( 1 - \exp - \psi^{(R)} \left(t, \zeta , \widehat{W} , \sigma \right) \right) \right) \right].
\end{align*}
This completes the proof.
\end{proof}

\bigskip

Fix $h_0 > 0$ and $\varepsilon > 0$. Let $\mathcal{P}^{(L,n,h_0,\varepsilon)}$ be the finite point measure on $[0,h_0] \times C(\mathbb{R}_+,\mathbb{R})^2 \times \mathbb{R}_+$ defined by
\begin{equation*}
\mathcal{P}^{(L,n,h_0,\varepsilon)} = \sum_{i \geq 0} \mathbf{1}_{\{ \frac{i}{n} \leq h_0 \}}
\mathbf{1}_{\{ \sigma ( C_{L_i}^{(n)} ) \geq \varepsilon \}}
\delta_{\frac{i}{n}} \otimes \delta_{( C_{L_i}^{(n)} , V_{L_i}^{(n)} )}
\otimes \delta_{\frac{2|L_i|}{n^2}}.
\end{equation*}
We denote by $\mathcal{P}^{(R,n,h_0,\varepsilon)}$ the point measure defined similarly for the right side of the spine. The random variables $\mathcal{P}^{(L,n,h_0,\varepsilon)}$ and $\mathcal{P}^{(R,n,h_0,\varepsilon)}$ take values in the space
\[E := \mathcal{M}_f \left( \mathbb{R}_+ \times C(\mathbb{R}_+,\mathbb{R})^2 \times \mathbb{R}_+ \right)\]
of all finite measures on $\mathbb{R}_+ \times C(\mathbb{R}_+,\mathbb{R})^2 \times \mathbb{R}_+$, which is a Polish space.

Let $Z$ be a nine-dimensional Bessel process started at $0$. As in the preceding proof we consider two point processes $\mathcal{P}^{(L)}$ and $\mathcal{P}^{(R)}$ on $\mathbb{R}_+ \times \Omega$, which conditionally given $Z$ are independent and Poisson with intensity given by \eqref{eq:pointsnakeintens}. Then we define a random element $\mathcal{P}^{(L,\infty,h_0,\varepsilon)}$ of $E$ by
$$
\int  \mathcal{P}^{(L,\infty,h_0,\varepsilon)} (\mathrm{d}h  \mathrm{d}f
\mathrm{d}g  \mathrm{d}s )
F(h,f,g,s) 
= \int \mathcal{P}^{(L)} (\mathrm{d}h \mathrm{d}\omega)
F(h, \zeta(\omega),Z_h+\widehat{W}(\omega), \sigma(\omega))
\mathbf{1}_{ \left\{ h \leq h_0 , \sigma(\omega) \geq \varepsilon \right\}}.
$$
We similarly define $\mathcal{P}^{(R,\infty,h_0,\varepsilon)}$ from the point process $\mathcal{P}^{(R)}$.

\begin{corollary}
\label{convjoint}
For every fixed $\varepsilon > 0$ and $h_0  > 0$,
\[
\left( X^{(n)} , \mathcal{P}^{(L,n,h_0,\varepsilon)} , \mathcal{P}^{(R,n,h_0,\varepsilon)} \right)
\underset{n \to \infty}{\longrightarrow}
\left(Z , \mathcal{P}^{(L,\infty,h_0,\varepsilon)} , \mathcal{P}^{(R,\infty,h_0,\varepsilon)}  \right),
\]
in the sense of convergence in distribution for random variables with values in $\mathbb{D}(\mathbb{R}_+) \times E \times E$.
\end{corollary}
\begin{proof}
Let us first show that the sequence of the laws of $\mathcal{P}^{(L,n,h_0,\varepsilon)}$ is tight. We will verify that, 
for every $\alpha>0$, there is a real number $M_\alpha\geq 0$ and a compact subset $K_\alpha$ of $[0,h_0]\times C(\mathbb{R}_+,\mathbb{R})^2 \times \mathbb{R}_+ $
such that, for every integer $n\geq 1$, with probability at least $1-\alpha$, the measure $\mathcal{P}^{(L,n,h_0,\varepsilon)}$ has total mass bounded by $M_\alpha$
and is supported on $K_\alpha$. Since the set of all finite measures supported on $K_\alpha$ with total mass bounded by $M_\alpha$
is compact, Prohorov's theorem will imply the desired tightness. 

Since for every $x\geq 1$, 
$$\widehat\rho_x(\sigma(C^{(n)}_\theta)\geq \varepsilon)\leq 2 \rho_x(\sigma(C^{(n)}_\theta)\geq \varepsilon)=
2 \rho_0(2|\theta|\geq \varepsilon n^2) =O(n^{-1})$$
a first moment calculation shows that we can find a constant $M_\alpha$ such that, for every $n\geq 1$,
\[
P \left[ \left| \mathcal{P}^{(L,n,h_0,\varepsilon)} \right| \geq M_{\alpha} \right] < \frac{\alpha}{2}.
\]

A similar argument shows the existence of a constant $H_\alpha$ large enough so that, for  every $n$,
\[P \left[  \mathcal{P}^{(L,n,h_0,\varepsilon)} ([0,h_0] \times C(\mathbb{R}_+,\mathbb{R})^2 \times ]H_\alpha,\infty[)>0\right]
 <\frac{\alpha}{4}.
\]
We will thus take the compact set $K_\alpha$ of the form
\[
K_\alpha= [0,h_0] \times {\mathcal K}_\alpha \times [0,H_\alpha].
\]
where ${\mathcal K}_\alpha$ will be a suitable compact subset of $C(\mathbb{R}_+,\mathbb{R})^2$. To construct ${\mathcal K}_\alpha$, we rely on the convergence results for discrete snakes. We first note that, thanks to the convergence in distribution of the rescaled processes $\left( \sqrt{\frac{3}{2n}}X_{\lfloor nt \rfloor} \right)_{t\geq 0}$, we can find a constant $A_\alpha$ such that, for every $n\geq 1$,
\[
P\left[\sup_{0\leq i\leq \lfloor h_0 n\rfloor} X_i \geq A_\alpha \sqrt{n} \right] < \alpha/8.
\]
Theorem 4 of \cite{CS}, or Theorem 2 of \cite{JM}, implies that the collection of the distributions of the processes
$(C^{(n)}_\theta,V^{(n)}_\theta)$ under the probability measures $\rho_{x} \left( \cdot\mid \varepsilon n^2\leq |\theta| \leq H_\alpha n^2 \right)$, for $n\geq 1$ and $x$ varying in $[0,A_\alpha\sqrt{n}]$, is tight (of course the choice of $x$ here just amounts to a translation of the labels). In particular, we can find compact subsets $\mathcal K$ of $C(\mathbb{R}_+,\mathbb{R})^2$ for which
\[
\rho_x\left((C^{(n)}_\theta,V^{(n)}_\theta)\notin {\mathcal K} \mid \varepsilon n^2\leq |\theta| \leq H_\alpha n^2\right)
\]
is arbitrarily small, uniformly in $x\in[0,A_\alpha\sqrt{n}]$ and $n\geq 1$. Using once again the bound $\widehat \rho_l\leq 2\rho_l$ and the estimate \eqref{lifetime2}, we can thus find a compact subset $\mathcal K_\alpha$ of $C(\mathbb{R}_+,\mathbb{R})^2$ such that
\[
(\lfloor nh_0\rfloor +1)\times \widehat\rho_x\left(\left\{(C^{(n)}_\theta,V^{(n)}_\theta)\notin {\mathcal K}_\alpha \right\}
 \cap\left\{\varepsilon n^2\leq |\theta| \leq H_\alpha n^2\right\}\right)
 \leq \alpha / 8,
\]
for every $x\in[0,A_\alpha\sqrt{n}]$ and $n\geq 1$. From this last bound and a first moment calculation, we get
\[
P\Big[\Big\{\sup_{0\leq i\leq \lfloor h_0 n\rfloor} X_i \leq A_\alpha \sqrt{n} \Big\}
 \cap \left\{\mathcal{P}^{(L,n,h_0,\varepsilon)} ([0,h_0] \times {\mathcal K}_\alpha^c \times [0,H_\alpha])>0\right\}\Big] \leq \alpha/8.
\]
We take $K_\alpha= [0,h_0] \times {\mathcal K}_\alpha \times [0,H_\alpha]$ as already mentioned, and by putting together the previous estimates, we arrive at
\[
P\left[\left\{\left| \mathcal{P}^{(L,n,h_0,\varepsilon)} \right| \leq M_{\alpha}\right\}
 \cap \left\{ \mathcal{P}^{(L,n,h_0,\varepsilon)}(K_\alpha^c)=0\right\}\right] \geq 1-\alpha.
\]
This completes the proof of tightness.

The same arguments also give the tightness of the sequence of the laws of $\mathcal{P}^{(R,n,h_0,\varepsilon)}$. Therefore, we know that the sequence of the laws of $\left( X^{(n)} , \mathcal{P}^{(L,n,h_0,\varepsilon)} , \mathcal{P}^{(R,n,h_0,\varepsilon)} \right)$ is tight.

Proposition \ref{pointconv}, and the remark following the statement of this proposition, now show that
\[
E \left[ \Psi \left( X^{(n)} ,\mathcal{P}^{(L,n,h_0,\varepsilon)}, \mathcal{P}^{(R,n,h_0,\varepsilon)} \right) \right]
\underset{n \to \infty}{\longrightarrow}
E \left[ \Psi \left( Z ,\mathcal{P}^{(\infty,h_0,\varepsilon)}_L, \mathcal{P}^{(\infty,h_0,\varepsilon)}_R \right) \right]
\]
for all functions $\Psi$ of the type
\[
\Psi(u , m_1 , m_2) = \phi(u) \exp \left( - \int \psi^{(L)} \, \mathrm{d}m_1 - \int \psi^{(R)} \, \mathrm{d}m_2 \right),
\]
with $\phi$, $\psi^{(L)}$ and $\psi^{(R)}$ as in Proposition \ref{pointconv}. Once we know that the sequence of the laws of $\left( X^{(n)} , \mathcal{P}^{(L,n,h_0,\varepsilon)} , \mathcal{P}^{(R,n,h_0,\varepsilon)} \right)$ is tight, this suffices to get the statement of Corollary \ref{convjoint}. 
\end{proof}

\begin{proof}[Proof of Theorem \ref{snakeinfinite}] 
Throughout the proof, $h_0>0$ is fixed. We consider as previously a triplet 
$(Z,\mathcal{P}^{(L)},\mathcal{P}^{(R)})$ such that $Z$ is a nine-dimensional Bessel process started at $0$, and conditionally given $Z$, $(\mathcal{P}^{(L)},\mathcal{P}^{(R)})$ is a pair of independent Poisson point processes on $\mathbb{R}_+ \times \Omega$ with intensity given by \eqref{eq:pointsnakeintens}. We assume that the process $W^{(L)}$, resp. $W^{(R)}$ is then 
determined from the pair $(Z,\mathcal{P}^{(L)})$, resp. $(Z,\mathcal{P}^{(R)})$, in the way
explained in subsect. \ref{subsec:infinitesnake}. In agreement with this subsection, we also use the notation
\[\tau^{(L)}_u=\sup \left\{s\geq 0: \zeta^{(L)}_s\leq u \right\}\]
for every $u\geq 0$.

Let us fix $\varepsilon>0$.
For every $n > 0$, let $C^{(L,n,h_0,\varepsilon)}$ denote the concatenation of the functions $\left( \frac{i}{n} + C_{L_i}^{(n)}(t) \right)_{0 \leq t < 2n^{-2} |L_i|}$, for all integers $i$ such that $2 n^{-2} |L_i| >\varepsilon$ and $i \leq n h_0$. The random function
$C^{(L,n,h_0,\varepsilon)}$ is defined and c\`adl\`ag on
the time interval $[0,\tau^{(L,n,h_0,\varepsilon)}[$, where
\begin{equation}
\label{tauneps}
\tau^{(L,n,h_0,\varepsilon)} = \sum_{i \leq n h_0} \mathbf{1}_{\{ 2 n^{-2} |L_i| > \varepsilon \}} 2 n^{-2} |L_i|.
\end{equation}
We extend the function $t\to C^{(L,n,h_0,\varepsilon)}$ to $[0,\infty[$ by setting 
$C^{(L,n,h_0,\varepsilon)}(t)=\frac{\lfloor nh_0\rfloor}{n}$ for every $t\in[\tau^{(L,n,h_0,\varepsilon)},\infty[$.

We denote the rescaled contour function of the left side of the spine of the uniform infinite well-labeled tree, up to
and including  its subtree $L_{\lfloor nh_0 \rfloor}$ at generation $\lfloor n h_0 \rfloor$, by $C^{(L,n,h_0)}$. The function $t\to  C^{(L,n,h_0)}(t)$ is defined and continuous over $[0,\tau^{(L,n,h_0)}]$, where 
as previously
\begin{equation}
\label{taun}
\tau^{(L,n,h_0)} = \frac{\lfloor n h _0 \rfloor}{n^2} + \sum_{i \leq n h_0} 2 n^{-2} |L_i|.
\end{equation}
Again, we extend $C^{(L,n,h_0)}$ to $[0,\infty[$ by setting $C^{(L,n,h_0)}(t)=\frac{\lfloor nh_0\rfloor}{n}$
if $t\geq \tau^{(L,n,h_0)}$. Note that we have also
\[
\tau^{(L,n,h_0)} =\sup \left\{ t\geq 0: \frac{1}{n}C^{(L)}(n^2t)\leq \frac{\lfloor nh_0\rfloor}{n} \right\}
\]
and that $C^{(L,n,h_0)}(t)=  \frac{1}{n}C^{(L)}(n^2(t\wedge\tau^{(L,n,h_0)} ))$ for every $t\geq 0$. 
The difference between $C^{(L,n,h_0)}$ and $C^{(L,n,h_0,\varepsilon)}$ comes from the time spent on the spine by the contour of $\theta$ and 
the contribution of small trees. 
See Fig. \ref{contourepsilon} for an illustration of the processes $C^{(L,n,h_0)}$ and $C^{(L,n,h_0,\varepsilon)}$.

\begin{figure}[!t]
\begin{center}
\includegraphics[width=0.85\textwidth]{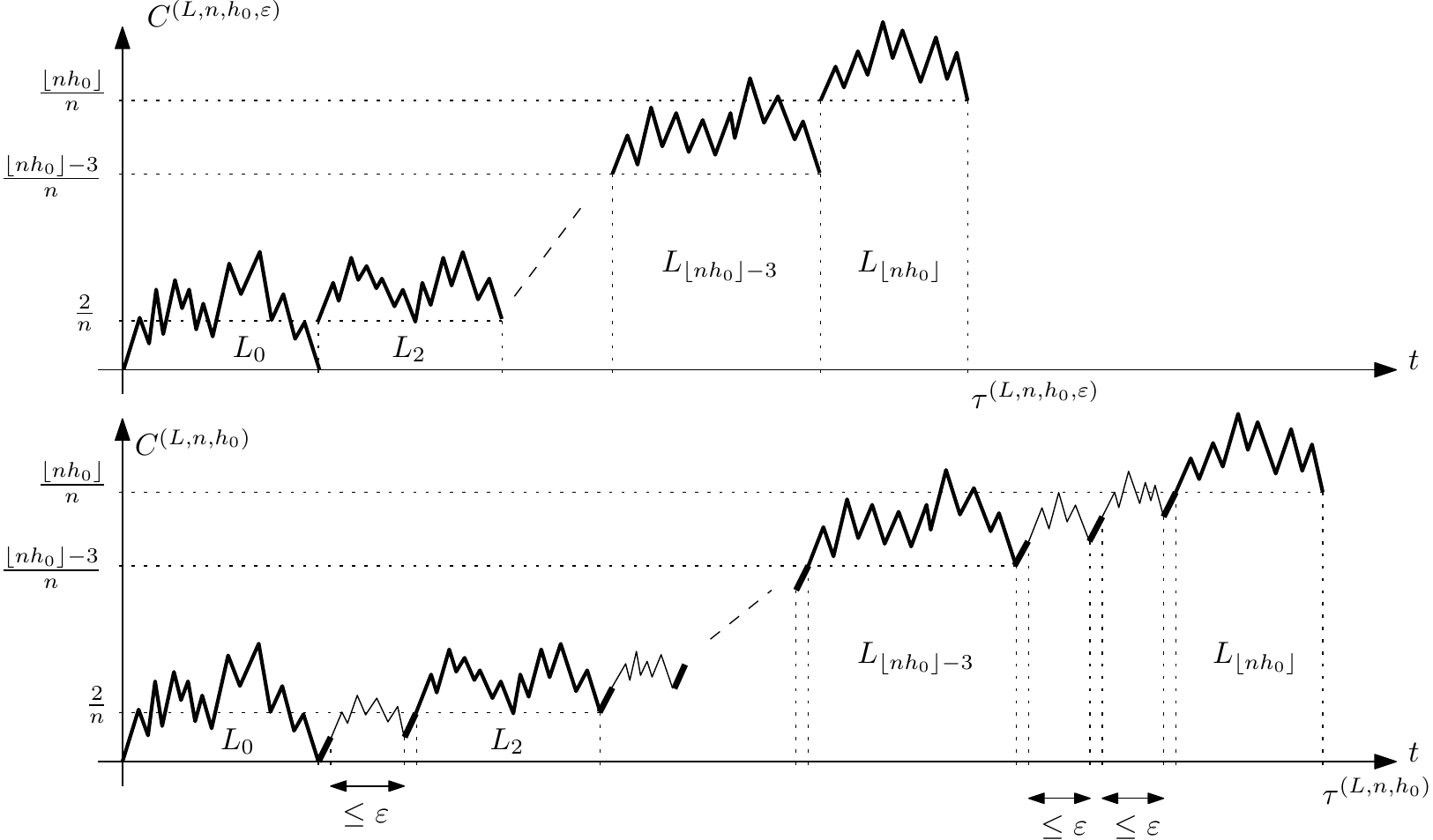}
\caption{The processes $C^{(L,n,h_0)}$ and $C^{(L,n,h_0,\varepsilon)}$.}
\label{contourepsilon}
\end{center}
\end{figure}

Similarly, we denote by $V^{(L,n,h_0,\varepsilon)}$ the concatenation of the functions $\left( V_{L_i}^{(n)}(t) \right)_{0 \leq t < 2n^{-2} |L_i|}$ for all integers $i$ such that $2n^{-2} |L_i| > \varepsilon$ and $i \leq n h_0$, and we extend this function 
to $[0,\infty[$ by setting $V^{(L,n,h_0,\varepsilon)}(t)= X^{(n)}_{\lfloor nh_0\rfloor/n}$
for $t\geq \tau^{(L,n,h_0,\varepsilon)}$. We define the process $V^{(L,n,h_0)}$
analogously to $C^{(L,n,h_0)}$, replacing the contour function by the spatial contour function.

We define in the same way the processes $C^{(R,n,h_0,\varepsilon)}$, $V^{(R,n,h_0,\varepsilon)}$, $C^{(R,n,h_0)}$ and $V^{(R,n,h_0)}$
for the right side of the spine.

Finally, let $\mathcal{P}^{(L,\infty,h_0,\varepsilon)}$ and $\mathcal{P}^{(R,\infty,h_0,\varepsilon)}$
be the point measures on ${\mathbb R}_+\times C({\mathbb R}_+,{\mathbb R})^2\times {\mathbb R}_+$
defined from $\mathcal{P}^{(L)}$ and $\mathcal{P}^{(R)}$ in the way explained before 
Corollary \ref{convjoint}.
We define four processes $C^{(L,\infty, h_0,\varepsilon)}$, $V^{(L,\infty, h_0,\varepsilon)}$, $C^{(R,\infty, h_0,\varepsilon)}$ and $V^{(R,\infty , h_0,\varepsilon)}$ by imitating the preceding construction but using the point measures $\mathcal{P}^{(L,\infty,h_0,\varepsilon)}$ and $\mathcal{P}^{(R,\infty,h_0,\varepsilon)}$ instead of $\mathcal{P}^{(L,n,h_0,\varepsilon)}$ and $\mathcal{P}^{(R,n,h_0,\varepsilon)}$. More explicitly, 
if $(r_1,(f_1,g_1),s_1)$, $(r_2,(f_2,g_2),s_2)$, etc. are the atoms of $\mathcal{P}^{(L,\infty,h_0,\varepsilon)}$ listed in such a 
way that $r_1<r_2<\cdots$, the process $C^{(L,\infty, h_0,\varepsilon)}$ is obtained by concatenating the functions 
$(r_1+f_1(t))_{0\leq t< s_1}$, $(r_2+f_2(t))_{0\leq t< s_2}$, etc., and  the process $V^{(L,\infty, h_0,\varepsilon)}$ is obtained by concatenating the functions 
$(g_1(t))_{0\leq t< s_1}$, $(g_2(t))_{0\leq t< s_2}$, etc. The random functions $C^{(L,\infty, h_0,\varepsilon)}$ and $V^{(L,\infty, h_0,\varepsilon)}$ are a priori only defined on a finite interval
$[0,\tau^{(L,\varepsilon)}_{h_0}[$, but we extend them to $[0,\infty[$ by setting
\[ \left( C^{(L,\infty, h_0,\varepsilon)}_t,V^{(L,\infty, h_0,\varepsilon)}_t \right) = \left( h_0,Z_{h_0} \right)\]
for every $t\geq \tau^{(L,\varepsilon)}_{h_0}$. 

\bigskip

Using Corollary \ref{convjoint} and the Skorokhod representation theorem, we may find, for every $n \geq 1$, a triplet $\left( \widetilde{X}^{(n)} , \widetilde{\mathcal{P}}^{(L,n,h_0,\varepsilon)} , \widetilde{\mathcal{P}}^{(R,n,h_0,\varepsilon)} \right)$ having the same law as the triplet $\left( X^{(n)} , \mathcal{P}^{(L,n,h_0,\varepsilon)} , \mathcal{P}^{(R,n,h_0,\varepsilon)} \right)$ and such that
\begin{equation}
\label{aspointconv}
\left( \widetilde{X}^{(n)} , \widetilde{\mathcal{P}}^{(L,n,h_0,\varepsilon)} , \widetilde{\mathcal{P}}^{(R,n,h_0,\varepsilon)} \right)
\underset{n \to \infty}{\longrightarrow}
\left( Z , \mathcal{P}^{(L,\infty,h_0,\varepsilon)} , \mathcal{P}^{(R,\infty,h_0,\varepsilon)} \right)
\end{equation}
almost surely. We can order the atoms of the point measures considered in \eqref{aspointconv} according to their first component.
From the convergence \eqref{aspointconv}, we deduce that almost surely for $n$ large enough 
the measures $\widetilde{\mathcal{P}}^{(L,n,h_0,\varepsilon)}$ and 
$\mathcal{P}^{(L,\infty,h_0,\varepsilon)}$ have the same number of atoms, and the $i$-th atom of $\widetilde{\mathcal{P}}^{(L,n,h_0,\varepsilon)}$
converges as $n\to\infty$ to the $i$-th atom of $\mathcal{P}^{(L,\infty,h_0,\varepsilon)}$. The same property holds for the right side of the spine. 

With the point measure $ \widetilde{\mathcal{P}}^{(L,n,h_0,\varepsilon)}$ , we can associate random functions $\widetilde C^{(L,n,h_0,\varepsilon)},
\widetilde V^{(L,n,h_0,\varepsilon)}$  defined in the same way as 
$C^{(L,n,h_0,\varepsilon)},
V^{(L,n,h_0,\varepsilon)}$ were defined from ${\mathcal{P}}^{(L,n,h_0,\varepsilon)}$.
Similarly, with the point measure $\widetilde{\mathcal{P}}^{(R,n,h_0,\varepsilon)}$
we associate the random functions $\widetilde C^{(R,n,h_0,\varepsilon)},
\widetilde V^{(R,n,h_0,\varepsilon)}$.
From the almost sure convergence of the atoms of $\widetilde{\mathcal{P}}^{(L,n,h_0,\varepsilon)}$, resp. $\widetilde{\mathcal{P}}^{(R,n,h_0,\varepsilon)}$, towards the corresponding atoms of $\mathcal{P}^{(L,\infty,h_0,\varepsilon)}$, resp. $\mathcal{P}^{(R,\infty,h_0,\varepsilon)}$, it is then an easy exercise, using the definition of the Skorokhod topology,
to check that we have almost surely
\begin{equation}
\label{eq:asconvsnake}
 \left(\widetilde C^{(L,n,h_0,\varepsilon)}, \widetilde V^{(L,n,h_0,\varepsilon)}\right) 
 \underset{n \to \infty}{\longrightarrow}
\left(C^{(L,\infty,h_0,\varepsilon)}, V^{(L,\infty,h_0,\varepsilon)} \right) \end{equation}
and similarly
\begin{equation}
\label{eq:asconvsnake2}
 \left(\widetilde C^{(R,n,h_0,\varepsilon)}, \widetilde V^{(R,n,h_0,\varepsilon)}\right) 
 \underset{n \to \infty}{\longrightarrow}
\left(C^{(R,\infty,h_0,\varepsilon)}, V^{(R,\infty,h_0,\varepsilon)} \right)
 \end{equation}
in the sense of the Skorokhod topology on ${\mathbb D}({\mathbb R}^2)$. 

Let $d_{\rm Sk}$ be a metric inducing the Skorokhod topology on ${\mathbb D}({\mathbb R}^2)$.
We may assume that $d_{\rm Sk}((f_1,g_1),(f_2,g_2))\leq \|f_1-f_2\|_\infty + \|g_1-g_2\|_\infty$, where 
$\|f\|_\infty=\sup\{|f(t)|:t\geq 0\}\leq \infty$. 

Then let $F$ be a bounded Lipschitz function on ${\mathbb D}({\mathbb R}^2)\times {\mathbb D}({\mathbb R}^2)$. 
From
(\ref{eq:asconvsnake}) and (\ref{eq:asconvsnake2}), we have
\begin{align}
\label{eq:cosnake}
&E\left[F\left( \left(C^{(L,n,h_0,\varepsilon)}, V^{(L,n,h_0,\varepsilon)}\right),
 \left( C^{(R,n,h_0,\varepsilon)}, V^{(R,n,h_0,\varepsilon)}\right) \right)\right]\notag\\
&\quad =E\left[F\left( \left(\widetilde C^{(L,n,h_0,\varepsilon)}, \widetilde V^{(L,n,h_0,\varepsilon)}\right),
 \left(\widetilde C^{(R,n,h_0,\varepsilon)}, \widetilde V^{(R,n,h_0,\varepsilon)}\right) \right)\right]
 \notag\\
 &\quad \underset{n \to \infty}{\longrightarrow} 
 E\left[F\left(\left(C^{(L,\infty,h_0,\varepsilon)}, V^{(L,\infty,h_0,\varepsilon)} \right),
 \left(C^{(R,\infty,h_0,\varepsilon)}, V^{(R,\infty,h_0,\varepsilon)} \right)\right)\right].
 \end{align}

Our goal is to prove that
 \begin{align}
\label{eq:cosnakebis}
&E\left[F\left( \left(C^{(L,n,h_0)}, V^{(L,n,h_0)}\right),
 \left(C^{(R,n,h_0)},V^{(R,n,h_0)}\right) \right)\right]
 \notag\\
 &\quad \underset{n \to \infty}{\longrightarrow} 
 E\left[F\left(\left(C^{(L,\infty,h_0)}, V^{(L,\infty,h_0)} \right),
 \left(C^{(R,\infty,h_0)}, V^{(R,\infty,h_0)} \right)\right)\right]
 \end{align}
where $(C^{(L,\infty,h_0)}(t), V^{(L,\infty,h_0)}(t))=(\zeta^{(L)}_{t\wedge \tau^{(L)}_{h_0}},
\widehat W^{(L)}_{t\wedge \tau^{(L)}_{h_0}})$, and the processes $(C^{(R,\infty,h_0)}(t), V^{(R,\infty,h_0)}(t))$ are defined in a similar manner. As we will explain later, the statement of Theorem \ref{snakeinfinite} easily follows from the convergence (\ref{eq:cosnakebis}).

In order to derive (\ref{eq:cosnakebis}) from (\ref{eq:cosnake}), we use the next lemma.

\begin{lemma}
\label{lemsnakeinfinite}
{\rm (i)} For every $\eta>0$, we have, for all $\varepsilon>0$ small enough,
\[\limsup_{n\to\infty} P\left[\sup_{t\geq 0} \left| C^{(L,n,h_0,\varepsilon)}(t)
 - C^{(L,n,h_0)}(t)\right| > \eta\right] < \eta\]
and
\[\limsup_{n\to\infty} P\left[\sup_{t\geq 0} \left| V^{(L,n,h_0,\varepsilon)}(t)
 - V^{(L,n,h_0)}(t)\right| > \eta\right] < \eta.\]
\noindent{\rm (ii)} We have for every $\eta >0$,
\[\lim_{\varepsilon\to 0} P\left[\sup_{t\geq 0} \left| C^{(L,\infty,h_0,\varepsilon)}(t)
 - C^{(L,\infty,h_0)}(t)\right| > \eta\right] =0\]
and
\[\lim_{\varepsilon\to 0} P\left[\sup_{t\geq 0} \left| V^{(L,\infty,h_0,\varepsilon)}(t)
 - V^{(L,\infty,h_0)}(t)\right| > \eta\right] =0.\]
\end{lemma}

Let us postpone the proof of Lemma \ref{lemsnakeinfinite} and complete the proof 
of Theorem \ref{snakeinfinite}. Fix $\delta>0$. From part (ii) of the lemma (and the obvious
analogue of this lemma for processes attached to the right side of the spine),
and our assumptions on $F$,
we can choose $\varepsilon_0>0$ such
that, for every $\varepsilon\in]0,\varepsilon_0[$,
\begin{align*}
&E\left[\left|F\left(\left(C^{(L,\infty,h_0)}, V^{(L,\infty,h_0)} \right),
 \left(C^{(R,\infty,h_0)}, V^{(R,\infty,h_0)} \right)\right)\right.\right.\\
&\qquad - \left.\left.F\left(\left(C^{(L,\infty,h_0,\varepsilon)}, V^{(L,\infty,h_0,\varepsilon)} \right),
 \left(C^{(R,\infty,h_0,\varepsilon)}, V^{(R,\infty,h_0,\varepsilon)} \right)\right)\right|\right] \leq \delta\,.
\end{align*}
From part (i) of the lemma, and choosing $\varepsilon$ even smaller if necessary, we have also
\begin{align*}
&\limsup_{n\to\infty} E\left[\left|F\left(\left(C^{(L,n,h_0)}, V^{(L,n,h_0)} \right),
 \left(C^{(R,n,h_0)}, V^{(R,n,h_0)} \right)\right)\right.\right.\\
&\qquad\qquad\qquad - \left.\left.F\left(\left(C^{(L,n,h_0,\varepsilon)}, V^{(L,n,h_0,\varepsilon)} \right),
 \left(C^{(R,n,h_0,\varepsilon)}, V^{(R,n,h_0,\varepsilon)} \right)\right)\right|\right] \leq \delta\,.
\end{align*}
Hence, using also (\ref{eq:cosnake}),
 \begin{align*}
&\limsup_{n\to\infty} E\left[\left|F\left(\left(C^{(L,n,h_0)}, V^{(L,n,h_0)} \right),
 \left(C^{(R,n,h_0)}, V^{(R,n,h_0)} \right)\right)\right.\right.\\
&\qquad\qquad\qquad - \left.\left.F\left(\left(C^{(L,\infty,h_0)}, V^{(L,\infty,h_0)} \right),
 \left(C^{(R,\infty,h_0)}, V^{(R,\infty,h_0)} \right)\right)\right|\right] \leq 2\delta\,.
 \end{align*}
Since $\delta$ was arbitrary, this completes the proof of (\ref{eq:cosnakebis}). 
We have thus obtained
\begin{align}
\label{snakeinfinitetech}
& \left(\left(C^{(L,n,h_0)}, V^{(L,n,h_0)} \right),
 \left(C^{(R,n,h_0)}, V^{(R,n,h_0)} \right)\right) \notag \\
& \qquad \qquad \underset{n\to\infty}{\overset{\mathrm{(d)}}{\longrightarrow}}
\left(\left(C^{(L,\infty,h_0)}, V^{(L,\infty,h_0)} \right),
 \left(C^{(R,\infty,h_0)}, V^{(R,\infty,h_0)} \right)\right).
 \end{align}
However, the pair $(C^{(L,n,h_0)}, V^{(L,n,h_0)})$ coincides with the process
$(\frac{1}{n}C^{(L)}(n^2\cdot),\sqrt{\frac{3}{2n}}V^{(L)}(n^2\cdot))$ stopped at time 
$\tau^{(L,n,h_0)}$, and the pair $(C^{(L,\infty,h_0)}, V^{(L,\infty,h_0)})$ coincides with
the process $(\zeta^{(L)},\widehat W^{(L)})$ stopped at time $\tau_{h_0}^{(L)}$.
Simple arguments (using the fact that \eqref{snakeinfinitetech} holds for every $h_0>0$)
show that $\tau^{(L,n,h_0)}$ must converge in distribution to $\tau_{h_0}^{(L)}$, and that
this convergence holds jointly with  \eqref{snakeinfinitetech}.

Analogous properties hold for the pairs $(C^{(R,n,h_0)}, V^{(R,n,h_0)})$
and $(C^{(R,\infty,h_0)}, V^{(R,\infty,h_0)})$, and for the random times 
$\tau^{(R,n,h_0)}$ and $\tau_{h_0}^{(R)}$ defined in an obvious manner
for the right side of the spine.
 Since $\tau^{(L)}_{h_0}$ and $\tau^{(R)}_{h_0}$ both increase to $\infty$
as $h_0\uparrow \infty$,
the statement
of Theorem \ref{snakeinfinite} follows from the convergence (\ref{snakeinfinitetech}). 
\end{proof}

\begin{proof}[Proof of Lemma \ref{lemsnakeinfinite}] 
We start by proving (ii). Write the atoms of ${\mathcal P}^{(L)}$ in the form
\[{\mathcal P}^{(L)}=\sum_{i\in I} \delta_{(r_i,\omega_i)}\]
and notice that, for every $u\geq 0$,
\[\tau^{(L)}_u=\sum_{i\in I} {\bf 1}_{\{r_i\leq u\}} \,\sigma(\omega_i).\]
The construction of $W^{(L)}$ from the point measure 
${\mathcal P}^{(L)}$ (cf subsect. \ref{subsec:infinitesnake}) shows that the pair
$(\zeta^{(L)},\widehat W^{(L)})$ is obtained by concatenating (in the appropriate order
given by the values of $r_i$) the functions
\[\left(r_i+\zeta_{\cdot}(\omega_i), Z_{r_i}+ \widehat W_\cdot(\omega_i)\right).\]
On the other hand, the definition of the point measure ${\mathcal P}^{(L,\infty,h_0,\varepsilon)}$,
and the construction of the pair $(C^{(L,\infty,h_0,\varepsilon)},V^{(L,\infty,h_0,\varepsilon)})$
from this point measure, show that the pair $(C^{(L,\infty,h_0,\varepsilon)},V^{(L,\infty,h_0,\varepsilon)})$
is obtained by concatenating the same functions, but only for those indices $i$
such that $r_i\leq h_0$ and $\sigma(\omega_i)\geq \varepsilon$. In other words, if we
define for every $t\geq 0$,
\[
A^{(L,h_0,\varepsilon)}_t = \int_0^t {\mathrm d}s\sum_{i\in I} {\bf 1}_{\{r_i\leq h_0,\sigma(\omega_i)\geq \varepsilon\}} \,{\bf 1}_{\{\tau^{(L)}_{r_i-}<s<\tau^{(L)}_{r_i}\}}
\]
and 
\[\gamma^{(L,h_0,\varepsilon)}_t =\inf \left\{ s\geq 0: A^{(L,h_0,\varepsilon)}_s >t \right\} \wedge 
\tau^{(L)}_{h_0},\]
we have 
\begin{equation}
\label{teclemsnake}
\left(C^{(L,\infty,h_0,\varepsilon)}(t),V^{(L,\infty,h_0,\varepsilon)}(t)\right)
=\left(\zeta^{(L)}_{\gamma^{(L,h_0,\varepsilon)}_t }, \widehat W^{(L)}_{\gamma^{(L,h_0,\varepsilon)}_t }
\right),
\end{equation}
for every $t\geq 0$. It is however immediate that
\[A^{(L,h_0,\varepsilon)}_t\underset{\varepsilon\to 0}{\longrightarrow} t\wedge \tau^{(L)}_{h_0}\]
and the convergence is uniform in $t$ by a monotonicity argument.
It follows that
\[\gamma^{(L,h_0,\varepsilon)}_t \underset{\varepsilon\to 0}{\longrightarrow} t\wedge \tau^{(L)}_{h_0}\]
again uniformly in $t$. Part (ii) of the lemma now follows from (\ref{teclemsnake}).

Let us turn to the proof of (i), which is more delicate. The general idea again is that
the process $C^{(L,n,h_0,\varepsilon)}$ can be written as a time change 
of $C^{(L,n,h_0)}$ (this should be obvious from Fig. \ref{contourepsilon}), and that this time change
is close to the identity when $\varepsilon$ is small. We start by estimating the
difference $\tau^{(L,n,h_0)} - \tau^{(L,n,h_0,\varepsilon)}$.
Let us fix $\delta > 0$. If $n$ is large enough so that 
$h_0/n<\delta/2$, we have, using \eqref{tauneps} and \eqref{taun},
\begin{align}
\label{comparetau}
P \left[ \tau^{(L,n,h_0)} - \tau^{(L,n,h_0,\varepsilon)} \geq \delta \right]
& = P \Big[ \frac{\lfloor n h _0 \rfloor}{n^2} + \sum_{i \leq n h_0} \mathbf{1}_{\{2 n^{-2} |L_i| \leq \varepsilon \}} 2 n^{-2} |L_i| \geq \delta \Big] \notag \\
&\leq  \frac{2}{\delta } E\Big[ \sum_{i \leq n h_0} \mathbf{1}_{\{2 n^{-2} |L_i| \leq \varepsilon \}} 2 n^{-2} |L_i|\Big]\notag\\
& = \frac{2}{\delta }
E \Big[
\sum_{i \leq n h_0}
\widehat{\rho}_{X_i} \left( \mathbf{1}_{\{2 n^{-2} |\theta| \leq \varepsilon \}} 2 n^{-2} |\theta| \right)
\Big] \notag \\
& \leq \frac{4 (\lfloor nh_0 \rfloor+1)}{\delta } \,
\rho_0 \left( \mathbf{1}_{\{2 n^{-2} |\theta| \leq \varepsilon \}} 2 n^{-2} |\theta| \right)\notag\\
& \leq K(h_0,\delta) \,\varepsilon^{1/2}
\end{align}
where the last bound is an easy consequence of (\ref{lifetime1}), with a constant $K(h_0,\delta)$
that depends only on $h_0$ and $\delta$. 

We now compare $C^{(L,n,h_0,\varepsilon)}$ and $C^{(L,n,h_0)}$. Note that we can write
$C^{(L,n,h_0,\varepsilon)}(t)=C^{(L,n,h_0)}(A_t)$, where the time change $A_t$ is such that $0\leq A_t-t\leq \tau^{(L,n,h_0)} - \tau^{(L,n,h_0,\varepsilon)}$ (a brief look at Fig. \ref{contourepsilon} should convince the reader).
It follows that
\begin{equation}
\label{timechangeC}
\sup_{t \geq 0}
\left| C^{(L,n,h_0,\varepsilon)} ( t ) -  C^{(L,n, h_0)}( t ) \right| 
\leq
\sup_{|t_1 - t_2| \leq \tau^{(L,n,h_0)} - \tau^{(L,n,h_0,\varepsilon)}}
\left| C^{(L,n,h_0)} \left( t_1  \right) - C^{(L,n,h_0)}  \left( t_2  \right)\right|.
\end{equation}

Recall that the function $C^{(L,n,h_0)}$  is constant on $[\tau^{(L,n,h_0)},\infty[$ by construction. 
In order to bound the left-hand side of \eqref{timechangeC},
we fix $t_1 \leq t_2 \leq \tau^{(L,n,h_0)}$ such that $t_2 - t_1 \leq \tau^{(L,n,h_0)} - \tau^{(L,n,h_0,\varepsilon)}$. 
If there exists $0 \leq i \leq n h_0$ such that
\[
\tau^{(L,n,(i-1)/n)} + n^{-2} \leq t_1 \leq t_2 < \tau^{(L,n,i/n)}+n^{-2},
\]
(with the convention $\tau^{(L,n,-1/n)} = - n^{-2}$) then this means that the times $t_1$ and $t_2$ correspond, in
the time scale of the rescaled contour process, to the exploration of the same tree $L_i$, or perhaps of the edge of the spine above
the root of $L_i$. In that case we can clearly bound
\begin{equation}
\label{contourLi}
\left| C^{(L,n,h_0)}(t_1) - C^{(L,n,h_0)}(t_2) \right| \leq
\sup_{|u - v| \leq \tau^{(L,n,h_0)} - \tau^{(L,n,h_0,\varepsilon)}}
\left| C_{L_i}^{(n)}(u) - C_{L_i}^{(n)}(v) \right| + \frac{1}{n}.
\end{equation}

On the other hand, if there exists no such $i$, then we can find $0 \leq i<j \leq n h_0$ such that
\[
\tau^{(L,n,(i-1)/n)} + n^{-2} \leq t_1  < \tau^{(L,n,i/n)} + n ^{-2} \leq \tau^{(L,n,(j-1)/n)} + n^{-2} \leq t_2 < \tau^{(L,n,j/n)} + n^{-2}
\]
and we have:
\begin{align*}
& \left| C^{(L,n,h_0)}(t_1) - C^{(L,n,h_0)}(t_2) \right| \\
&\quad  \leq \left| C_{L_j}^{(n)} \left(t_2 - \tau^{(L,n,(j-1)/n)} - n^{-2}\right)
- C_{L_i}^{(n)} \left(t_1 - \tau^{(L,n,(i-1)/n)} - n^{-2})\right)\right|
+ \frac{j-i+1}{n},
\end{align*}
where we recall the convention that $C^{(n)}_{L_i}(s)=0$ for $s\geq 2|L_i|/n^2$. 
Now note that $i=J_{\lfloor n^2t_1\rfloor}$ and $j=J_{\lfloor n^2t_2\rfloor}$,
with the notation introduced before Lemma \ref{techinfinite}.
We obtain
\begin{align}
\label{contourLiLj}
& \left| C^{(L,n,h_0)}(t_1) - C^{(L,n,h_0)}(t_2) \right| \notag\\
& \leq \frac{J_{\lfloor n^2t_2\rfloor}-J_{\lfloor n^2t_1\rfloor}+1}{n}
+ \max\left\{
C_{L_i}^{(n)} \left(t_1 - \tau^{(L,n,(i-1)/n)} + n^{-2}\right),
C_{L_j}^{(n)} \left(t_2 - \tau^{(L,n,(j-1)/n)} + n^{-2}\right)
\right\} 
\end{align}

Put $\gamma_{n,\varepsilon}= \tau^{(L,n,h_0)} - \tau^{(L,n,h_0,\varepsilon)}$
to simplify notation. From \eqref{timechangeC} and
the bounds \eqref{contourLi} and \eqref{contourLiLj}, we get
\begin{align}
\label{compareCn}
&\sup_{t \geq 0}
\left| C^{(L,n,h_0,\varepsilon)} ( t ) -  C^{(L,n, h_0)} ( t ) \right| \notag \\
&\ \leq
\sup_{u,v\leq \tau^{(L,n,h_0)}, |v-u|\leq \gamma_{n,\varepsilon}}\,
\frac{|J_{\lfloor n^2v\rfloor}-J_{\lfloor n^2u\rfloor}|+1}{n}
+ \sup_{0\leq k\leq \lfloor nh_0\rfloor}\;  \sup_{|v - u| \leq 
\gamma_{n,\varepsilon}} \left| C_{L_k}^{(n)}(v) - C_{L_k}^{(n)}(u) \right|.
\end{align}
We write $\beta_1(n,\varepsilon)$ and $\beta_2(n,\varepsilon)$ for the two terms in the sum of the
right-hand side of (\ref{compareCn}).
We will use Lemma \ref{techinfinite} to handle $\beta_1(n,\varepsilon)$, but we need a different argument for 
$\beta_2(n,\varepsilon)$. Recall our notation $H(\theta)$ for the height of a labeled tree $\theta$. Then, for every $\delta >0$ and $\kappa>0$,
\begin{align}
\label{lemsnake08}
&P\left[\sup_{0\leq k\leq \lfloor nh_0\rfloor}\;  \sup_{|u - v| \leq 
\delta} \left| C_{L_k}^{(n)}(u) - C_{L_k}^{(n)}(v) \right| >\kappa\right]\notag\\
&\qquad \leq \sum_{k=0}^{\lfloor nh_0\rfloor}
P\left[ \sup_{|u-v|\leq \delta} |C_{L_k}(n^2 u) - C_{L_k}(n^2v)| > n\kappa\right]\notag\\
&\qquad = \sum_{k=0}^{\lfloor nh_0\rfloor} E\left[\widehat\rho_{X_k}\left(
\sup_{|u-v|\leq \delta} |C_{\theta}(n^2 u) - C_{\theta}(n^2v)| > n\kappa\right)\right]\notag\\
&\qquad \leq 2(\lfloor nh_0\rfloor +1)\;\rho_0 \left(
\sup_{|u-v|\leq \delta} |C_{\theta}(n^2 u) - C_{\theta}(n^2v)| > n\kappa\right)\notag\\
&\qquad = 2(\lfloor nh_0\rfloor +1)\;\rho_0 (H(\theta)>n\kappa)\times
\rho_0 \left(
\sup_{|u-v|\leq \delta} |C^{(n)}_{\theta}(u) - C^{(n)}_{\theta}(v)| > \kappa
\,\Big| \,H(\theta)>n\kappa\right).
\end{align}
By standard results about Galton-Watson trees, 
\begin{equation}
\label{bounduniheight}
\sup_{n\geq 1} n \,\rho_0(H(\theta) \geq n) <\infty
\end{equation}
and so the quantities $2(\lfloor nh_0\rfloor +1)\;\rho_0(H(\theta)>n\kappa)$
are bounded above by a constant $K(h_0,\kappa)$ depending only on $h_0$ and $\kappa$.
On the other hand, from Corollary 1.13 in \cite{randomtrees} (or as an easy consequence
of Proposition \ref{convtoBS}), the law of $(C^{(n)}_\theta(t))_{0\leq t\leq 2\,n^{-2}|\theta|}$ under the conditional
probability measure $\rho_0(\cdot\mid H(\theta) > n\kappa)$ converges as $n\to\infty$
to the law of a Brownian excursion with height greater than $\kappa$. Consequently,
\begin{align*}
&\limsup_{n\to\infty} \rho_0 \left(
\sup_{|u-v|\leq \delta} |C^{(n)}_{\theta}(u) - C^{(n)}_{\theta}(v)| > \kappa\,
\Big|\, H(\theta)>n\kappa\right)\\
&\qquad\leq {\mathbf n}\left( \sup_{|u-v|\leq \delta} |e(u)-e(v)| \geq \kappa\,\Big|\, \sup_{t\geq 0} e(t) \geq \kappa\right),
\end{align*}
where ${\mathbf n}$ stands for the It\^o excursion measure as in subsect. \ref{subsec:browniansnake}. For any fixed $\kappa$, the right-hand side
can be made arbitrarily small by choosing $\delta$ small enough.

To complete the argument, fix $\eta >0$. By the preceding considerations, we can choose $\delta>0$ small enough so that
\begin{equation}
\label{lemsnake11}
\limsup_{n\to\infty} 
P\left[\sup_{0\leq k\leq \lfloor nh_0\rfloor}\;  \sup_{|u - v| \leq 
\delta} \left| C_{L_k}^{(n)}(u) - C_{L_k}^{(n)}(v) \right| >\frac{\eta}{2}\right] < \frac{\eta}{3}.
\end{equation}
and, using Lemma \ref{techinfinite},
\begin{equation}
\label{lemsnake12}
\limsup_{n\to\infty} P\left[\sup_{u,v\leq \tau^{(L,n,h_0)}\;,\; |v-u| \leq \delta}\,
\frac{|J_{\lfloor n^2v\rfloor}-J_{\lfloor n^2u\rfloor}|+1}{n}> \frac{\eta}{2}\right] < \frac{\eta}{3}.
\end{equation}
From (\ref{compareCn}), we get
\begin{align*}
&P\left[\sup_{t \geq 0}
\left| C^{(L,n,h_0,\varepsilon)} ( t ) -  C^{(L,n, h_0)} ( t ) \right| >\eta\right]\\
&\qquad\leq P \left[ \gamma_{n,\varepsilon}\geq \delta \right] + P \left[ \gamma_{n,\varepsilon}< \delta\,,\,\beta_1(n,\varepsilon)>\frac{\eta}{2} \right]
+ P \left[ \gamma_{n,\varepsilon}< \delta\,,\,\beta_2(n,\varepsilon)>\frac{\eta}{2} \right].
\end{align*}
The quantities $P[\gamma_{n,\varepsilon}< \delta\,,\,\beta_1(n,\varepsilon)>\frac{\eta}{2}]$
and $P[\gamma_{n,\varepsilon}< \delta\,,\,\beta_2(n,\varepsilon)>\frac{\eta}{2}]$ are smaller than $\frac{\eta}{3}$ when $n$ is large
(independently of the choice of $\varepsilon$), by (\ref{lemsnake11}) and (\ref{lemsnake12}). Finally, (\ref{comparetau})
allows us to choose $\varepsilon >0$ sufficiently small so that $P[\gamma_{n,\varepsilon}\geq \delta]< \frac{\eta}{3}$ for every
$n\geq 1$. This completes the proof of the first assertion in (i). 

The second assertion in (i) is proved in a similar way, and we only point at the differences. The same arguments 
we used to obtain the bound (\ref{compareCn}) give
\begin{align}
\label{compareVn}
&\sup_{t \geq 0}
\left| V^{(L,n,h_0,\varepsilon)} ( t ) -  V^{(L,n, h_0)} ( t ) \right| \notag \\
&\ \leq
\sup_{u,v\leq \tau^{(L,n,h_0)}, |v-u|\leq \gamma_{n,\varepsilon}}\,\sqrt{\frac{3}{2n}}\,
\left(|X_{J_{\lfloor n^2v\rfloor}}-X_{J_{\lfloor n^2u\rfloor}}|+1\right)
+ \sup_{0\leq k\leq \lfloor nh_0\rfloor}\;  \sup_{|v - u| \leq 
\gamma_{n,\varepsilon}} \left| V_{L_k}^{(n)}(v) - V_{L_k}^{(n)}(u) \right|.
\end{align}
If $\eta>0$ is fixed, we can again use Lemma \ref{techinfinite}, together with 
Proposition \ref{Bes9}, to see that we can choose $\delta>0$ small enough so that
\begin{equation}
\label{lemsnake22}
\limsup_{n\to\infty} P\left[\sup_{ u,v\leq \tau^{(L,n,h_0)},|v-u| \leq \delta}\,
\,\sqrt{\frac{3}{2n}}\,
\left(|X_{J_{\lfloor n^2v\rfloor}}-X_{J_{\lfloor n^2u\rfloor}}|+1\right)> \frac{\eta}{2}\right] < \frac{\eta}{3}.
\end{equation}
Then, in order to estimate the second term of the right-hand side of (\ref{compareVn}), we replace the
bound (\ref{lemsnake08}) by
\begin{align}
\label{lemsnake18}
&P\left[\sup_{0\leq k\leq \lfloor nh_0\rfloor}\;  \sup_{|u - v| \leq 
\delta} \left| V_{L_k}^{(n)}(u) - V_{L_k}^{(n)}(v) \right| >\kappa\right]\notag\\
&\qquad \leq 2(\lfloor nh_0\rfloor +1)\;\rho_0 (V^{**}(\theta)>\frac{\kappa}{2}\sqrt{n})\times
\rho_0 \left(
\sup_{|u-v|\leq \delta} |V^{(n)}_{\theta}(u) - V^{(n)}_{\theta}(v)| > \kappa
\,\Big| \,V^{**}(\theta)>\frac{\kappa}{2}\sqrt{n}\right),
\end{align}
where $V^{**}(\theta)$ denotes the maximal absolute value of a label in $\theta$. The analogue of
(\ref{bounduniheight}) is 
\begin{equation}
\label{boundunilabel}
\sup_{n\geq 1} n \,\rho_0(V^{**}(\theta)\geq \sqrt{n}) < \infty.
\end{equation}
This bound can be derived from the much more precise estimate given
in Proposition 4 of \cite{CS} (together with (\ref{lifetime1})). Then,  Proposition \ref{convtoBS}
implies that the law of $( V^{(n)}_\theta(t) )_{0\leq t\leq 2\, n^{-2}|\theta|}$ under the conditional
probability measure $\rho_0(\cdot\mid V^{**}(\theta) > \frac{\kappa}{2}\sqrt{n})$ converges as $n\to\infty$
to the law of $( \widehat W_s)_{0\leq t\leq \sigma}$ under ${\mathbb N}_0 ( \cdot \mid W^{**}> (3/8)^{1/2}\kappa)$, where
$W^{**}=\max \{ |\widehat W_s|:s\geq 0 \}$ (the precise justification of this convergence uses arguments very similar to the proof of Corollary 1.13 in \cite{randomtrees}).
Consequently,
\begin{align*}
&\limsup_{n\to\infty} \rho_0 \left(
\sup_{|u-v|\leq \delta} |V^{(n)}_{\theta}(u) - V^{(n)}_{\theta}(v)| > \kappa
\,\Big| \,V^{**}(\theta)>\frac{\kappa}{2}\sqrt{n}\right)\\
&\qquad\leq {\mathbb N}_0\left( \sup_{|u-v|\leq \delta} |\widehat W(u)-\widehat W(v)| \geq \kappa\,\Big|\,  W^{**}> (3/8)^{1/2}\kappa\right),
\end{align*}
and, for any fixed $\kappa>0$, the left-hand side can be made arbitrarily small by choosing $\delta$ small. 
The remaining part of the proof is exactly similar to the proof of the first assertion in (i). This completes the
proof of Lemma \ref{lemsnakeinfinite}.
\end{proof}

\section{Distances in  the uniform infinite quadrangulation}

\label{sec:scalequad}

The main result of this section provides a scaling limit for the profile of distances
in the uniform infinite quadrangulation. In order to derive this result
from Theorem \ref{snakeinfinite}, we need a
preliminary lemma. We use the same notation as in Theorem \ref{snakeinfinite}.

\begin{lemma}
\label{truncat-dist}
 Let $A>0$. We have
$$\lim_{K\to\infty} \Big(\sup_{n\geq 1} P\Big[ \inf_{t\geq K} V^{(L)}(n^2t) < A\sqrt{n}\Big]\Big)=0.$$
\end{lemma}

\begin{proof}
We first note that for every fixed $n\geq 1$, the probability considered
in the lemma tends to $0$ as $K\to\infty$ because $V^{(L)}(k)$ tends to $\infty$ as $k\to\infty$.
The problem is thus to get uniformity in $n$, and for this purpose we may restrict our attention to
values of $n$ that are larger than some fixed constant. 

Next we observe that it is enough to prove that
$$\lim_{h\to\infty} \Big(\sup_{n\geq 1} P\Big[ \inf_{t\geq \tau^{(L,n,h)}} V^{(L)}(n^2t) < A\sqrt{n}\Big]\Big)=0.$$
Indeed, since we know that $\tau^{(L,n,h)}$ converges in distribution towards 
$\tau^{(L)}_h$ as $n\to\infty$, with $\tau^{(L)}_h<\infty$ a.s., we can for every fixed value of 
$h>0$ choose $K$ sufficiently large so that $P[\tau^{(L,n,h)}>K]$ is arbitrarily small,
uniformly in $n$. Thus the probability in the lemma will be bounded above by the probability 
appearing in the last display, up to a (uniform in $n$) small error.

The event
$$\Big\{ \inf_{t\geq \tau^{(L,n,h)}} V^{(L)}(n^2t) < A\sqrt{n}\Big\}$$
may occur only if one of the trees $L_i, i\geq \lfloor nh\rfloor$ has a vertex with label
smaller than $A\sqrt{n}$. Hence the probability of the complement of this event 
is bounded below by
$$E\Big[ \prod_{i=\lfloor nh\rfloor}^\infty \widehat \rho_{X_i}(V_*\geq A\sqrt{n})\Big]$$
where we recall our notation $V_*$ for the minimal label in a labeled tree $\theta$. 
The preceding quantity can also be written in the form
\begin{equation}
\label{trucat-tech}
E\Big[\exp\sum_{i=\lfloor nh\rfloor}^\infty \log(1-\widehat \rho_{X_i}(V_*< A\sqrt{n}))\Big]
\end{equation}

Let us fix $\varepsilon \in]0,1/4[$, and set $B=64A/\varepsilon^2$. Consider the event
$$\Gamma_{h,n}=\{ X_i> B\sqrt{n}\;,\hbox{ for every }i\geq \lfloor nh\rfloor\}.$$
As a consequence of Proposition \ref{Bes9} and Lemma 2
in \cite{Me}, we can choose $h>0$ large enough so that, for every 
sufficiently large $n$,  $P[\Gamma_{h,n}]>1-
\varepsilon$. We will prove that, for this value of $h$, and for every sufficiently large
$n$, the quantity in \eqref{trucat-tech} is bounded below by $1-3\varepsilon$. This will complete
the proof of the lemma.

To get a lower bound on the quantity \eqref{trucat-tech}, we recall from Section 2 that,
for every $l\geq 1$,
$$\rho_l(V_*>0)= \frac{l(l+3)}{(l+1)(l+2)} = 1 -\frac{2}{(l+1)(l+2)}.$$
Since $\rho_l(V_*\geq 0)=\rho_l(V_*>-1)=\rho_{l+1}(V_*>0)$, it follows that, for every $l\geq 1$,
$$\rho_l(V_*=0)=\frac{4}{(l+1)(l+2)(l+3)}\leq \frac{4}{l^3}.$$
Note that $\rho_l(V_*=l')=\rho_{l-l'}(V_*=0)$ if $l>l'\geq 0$. If $X_i> B \sqrt{n}$, we have thus
$$\widehat \rho_{X_i}(V_*<A\sqrt{n})
\leq 2\,\rho_{X_i}(0<V_*<A\sqrt{n})
\leq \frac{8\lfloor A\sqrt{n}\rfloor}{(X_i-A\sqrt{n})^3}\leq \frac{16\lfloor A\sqrt{n}\rfloor}{X_i^3}.$$
Hence, on the event $\Gamma_{h,n}$, for $n$ sufficiently large, we have
$$\Big|\sum_{i=\lfloor nh\rfloor}^\infty \log(1-\widehat \rho_{X_i}(V_*< A\sqrt{n}))\Big|
\leq 2 \sum_{i=\lfloor nh\rfloor}^\infty \frac{16 A\sqrt{n}}{X_i^3}.$$
For every integer $j\geq 1$, set $\Delta_j=\#\{i\geq 0: X_i=j\}$. By Proposition
5.1 in \cite{CD}, we have $E[\Delta_j]\leq j$, for all sufficiently large $j$. Hence, if
$n$ is sufficiently large,
\begin{align*}
E\Big[ {\bf 1}_{\Gamma_{h,n}} \sum_{i=\lfloor nh\rfloor}^\infty \frac{32 A\sqrt{n}}{X_i^3}\Big]
&\leq E\Big[\sum_{i=0}^\infty \frac{32 A\sqrt{n}}{X_i^3}\,{\bf 1}_{\{X_i> B\sqrt{n}\}}\Big]\\
&= 32A\sqrt{n}\;E\Big[\sum_{j=\lfloor B\sqrt{n}\rfloor +1}^\infty \frac{1}{j^3}\,\Delta_j\Big]\\
&\leq 32A\sqrt{n} \sum_{j=\lfloor B\sqrt{n}\rfloor +1}^\infty \frac{1}{j^2}\\
&\leq 64A/B\\
&\leq \varepsilon^2,
\end{align*}
by our choice of $B$. Using the Markov inequality, we now get
$$P\Big[ \Gamma_{h,n}\cap \Big\{ \Big|\sum_{i=\lfloor nh\rfloor}^\infty \log(1-\widehat \rho_{X_i}(V_*< A\sqrt{n}))\Big|>\varepsilon\Big\}\Big]\leq \varepsilon.$$
Recalling that $P[\Gamma_{n,h}]>1-\varepsilon$, we thus see that the quantity 
inside the expectation in (\ref{trucat-tech}) is bounded below by $\exp(-\varepsilon)\geq 1-\varepsilon$,
except possibly on an event of probability at most $2\varepsilon$. It follows that the 
quantity (\ref{trucat-tech}) is bounded below by $1-3\varepsilon$, which was the desired result.
\end{proof}

Recall that the profile $\lambda_q$ of a quadrangulation $q$ is the integer-valued measure on $\mathbb{Z}_+$ defined by
\[
\lambda_q(k) = \left| \left\{ a \in V(q) : \, d_{gr}(\partial,a) = k \right\} \right|
\]
for every $k \in \mathbb{Z}_+$. If $q \in \overline{\mathbf Q}$ and $n \geq 1$ is an integer, we define the rescaled profile $\lambda_q^{(n)}$ as the $\sigma$-finite measure on $\mathbb{R}_+$ such that
\[
\lambda_q^{(n)} (A) = \frac{1}{n^2} \lambda_q \left(
  \sqrt{\frac{2n}{3}} A \right)
\]
for any Borel subset $A$ of $\mathbb{R}_+$. Also recall that $B_{n}({\bf q})$ denotes the ball of radius $n$
centered at $\partial$ in $V({\bf q})$

\begin{theorem}
\label{profile}
Let  ${\mathbf q}$ be a uniform infinite quadrangulation.
The sequence  $( \lambda_{\mathbf q}^{(n)} )_{n \geq 1}$ converges in distribution to the random measure $\mathcal{I}$ on ${\mathbb R}_+$, which is defined, for every continuous function $g$ with compact support, by
\[
\langle \mathcal{I},g \rangle = \frac{1}{2}\int_0^{\infty} \mathrm{d}s
\left(g \left( \widehat{W}_s^{(L)} \right) + g \left( \widehat{W}_s^{(R)} \right)\right)
\]
where $\left( W^{(L)} , W^{(R)} \right)$ is a pair of correlated eternal conditioned Brownian snakes.

In particular we have:
\[
\frac{1}{n^4} \# B_{n } ({\mathbf q}) 
\underset{n \to \infty}{\overset{(d)}{\longrightarrow}}
\frac{9}{4}\, {\mathcal I}([0,1]).
\]
\end{theorem}
\begin{remark}
Both $\lambda_{{\mathbf q}}^{(n)}$ and $\mathcal{I}$ are random variables with values in the space of Radon measures on $\mathbb{R}_+$, which is a Polish space for the topology of vague convergence. The convergence in distribution of the sequence $(\lambda^{(n)}_{\bf q})_{n\geq 1}$ thus refers to this topology.
\end{remark}
\begin{proof}
We may assume that ${\mathbf q}$ is the image under the extended Schaeffer correspondence of 
a uniform infinite well-labeled tree $\Theta$, and we use the same notation $(X_i,L_i,R_i)_{i \geq 0}$ as in subsect. \ref{subsec:prooftree}.
For every $i \geq 0$, we write the labeled trees $L_i$ and $R_i$ as $L_i = (\tau_{L_i},\ell_{L_i})$ and $R_i = (\tau_{R_i},\ell_{R_i})$. 
We also keep the notation $(C^{(L)},V^{(L)})$, resp. $(C^{(R)},V^{(R)})$, for the pair of contour functions coding the part of $\Theta$
to the left of the spine,
resp. to the right of the spine.

Fix a continuous function $g$ with compact support on ${\mathbb R}_+$.
From the properties of the Schaeffer correspondence, we have then
\begin{equation}
\label{formulaprofile}
\langle \lambda_{\mathbf q} , g \rangle = g(0) + \sum_{i=0}^\infty g(X_i) + \sum_{i =0}^\infty
\left( \sum_{v \in {L_i\backslash\{\emptyset\}}} g( {\ell_{L_i}(v)} )
+ \sum_{v \in {R_i\backslash\{\emptyset\}}} g( {\ell_{R_i}(v)}) \right).
\end{equation}
We can rewrite the right-hand side of \eqref{formulaprofile} in terms of the contour functions of 
$\Theta$. To this end, set for every $t\geq 0$, $[t]_{C^{(L)}}=\lfloor t\rfloor +1$ if $C^{(L)}(\lfloor t\rfloor +1)> C^{(L)}(\lfloor t\rfloor)$, and 
$[t]_{C^{(L)}}=\lfloor t\rfloor$ otherwise. Define $[t]_{C^{(R)}}$ in a similar way. Then, from the construction of the contour functions, it is easy to verify that we have also
\begin{equation}
\label{formulaprofilebis}
\langle \lambda_{\mathbf q} , g \rangle = g(0) + g(1) + \frac{1}{2} \int_0^\infty \mathrm{d}t\,g(V^{(L)}([t]_{C^{(L)}}))
+  \frac{1}{2}  \int_0^\infty \mathrm{d}t\,g(V^{(R)}([t]_{C^{(R)}})).
\end{equation}
Consequently,
$$
\langle \lambda^{(n)}_{\mathbf q} , g \rangle = \frac{g(0) + g(\sqrt{\frac{3}{2n}})}{n^2} + \frac{1}{2} \int_0^\infty \mathrm{d}t\,g\Big(\sqrt{\frac{3}{2n}}V^{(L)}([n^2t]_{C^{(L)}})\Big)
+  \frac{1}{2}  \int_0^\infty \mathrm{d}t\,g\Big(\sqrt{\frac{3}{2n}}V^{(R)}([n^2t]_{C^{(R)}})\Big).
$$
Since $|V^{(L)}([s]_{C^{(L)}})-V^{(L)}(s)|\leq 1$, for every $s\geq 0$, and $g$ is compactly supported hence uniformly continuous, a simple
argument, using also Lemma \ref{truncat-dist}, shows that 
$$\int_0^\infty \mathrm{d}t\,g\Big(\sqrt{\frac{3}{2n}}V^{(L)}([n^2t]_{C^{(L)}})\Big)
- \int_0^\infty \mathrm{d}t\,g\Big(\sqrt{\frac{3}{2n}}V^{(L)}(n^2t)\Big) \build{\longrightarrow}_{n\to\infty}^{(P)} 0,$$
where the notation $\build{\longrightarrow}_{}^{(P)}$ indicates convergence in probability.
Thus we have obtained
\begin{equation}
\label{rescapro}
\langle \lambda^{(n)}_{\mathbf q} , g \rangle
-\frac{1}{2}\Big( \int_0^\infty \mathrm{d}t\,g\Big(\sqrt{\frac{3}{2n}}V^{(L)}(n^2t)\Big) + 
\int_0^\infty \mathrm{d}t\,g\Big(\sqrt{\frac{3}{2n}}V^{(R)}(n^2t)\Big) \build{\longrightarrow}_{n\to\infty}^{(P)} 0.
\end{equation}
By Lemma \ref{truncat-dist},
\begin{equation}
\label{truncat-pro}
P\Big[ \int_0^\infty \mathrm{d}t\,g\Big(\sqrt{\frac{3}{2n}}V^{(L)}(n^2t)\Big)=\int_0^K \mathrm{d}t\,g\Big(\sqrt{\frac{3}{2n}}V^{(L)}(n^2t)\Big)\Big]
\build{\longrightarrow}_{K\to\infty}^{} 1,
\end{equation}
uniformly in $n\geq 1$, and a similar result holds for the integrals involving $V^{(R)}$. Moreover, by \eqref{transieternal},
\begin{equation}
\label{truncat-probis}
P\Big[ \langle {\mathcal I}, g\rangle = 
\frac{1}{2}\int_0^{K} \mathrm{d}s
\left(g \left( \widehat{W}_s^{(L)} \right) + g \left( \widehat{W}_s^{(R)} \right)\right)\Big]
\build{\longrightarrow}_{K\to\infty}^{} 1.
\end{equation}
Theorem \ref{snakeinfinite} implies that, for every $K\geq 0$, 
$$\int_0^K \mathrm{d}t\,\Big(g\Big(\sqrt{\frac{3}{2n}}V^{(L)}(n^2t)\Big)+ g\Big(\sqrt{\frac{3}{2n}}V^{(R)}(n^2t)\Big)\Big)
 \build{\longrightarrow}_{n\to\infty}^{(d)}\int_0^{K} \mathrm{d}s
\left(g \left( \widehat{W}_s^{(L)} \right) + g \left( \widehat{W}_s^{(R)} \right)\right).
$$
From this convergence, \eqref{rescapro}, \eqref{truncat-pro} and \eqref{truncat-probis}, we get that 
$\langle \lambda^{(n)}_{\mathbf q} , g \rangle$ converges in distribution to $\langle {\mathcal I}, g\rangle$,
which completes the proof of the first assertion.

\smallskip

Note that $ \mathcal{I}([0,r])\build{=}_{}^{(d)} r^4 \mathcal{I}([0,1])$
for every $r>0$, by a simple scaling argument. Since
\[
\frac{1}{n^4} \# B_{n}(\mathbf{q})  = \lambda_{\mathbf q}^{(n^2)}([0,(3/2)^{1/2}]),
\]
the second assertion of the theorem will follow if we can verify that $\lambda_{\mathbf q}^{(n)}([0,r])$
converges in distribution to $\mathcal{I}([0,r])$ for every $r>0$. This is a straightforward consequence of the first assertion and the fact that $\mathcal{I}(\{r\}) = 0$ a.s. The latter fact is easy from a first-moment calculation.
\end{proof}

\medskip

Let us conclude with some remarks about the distribution of $\mathcal{I} \left( [0,1] \right)$. 
From the definition of the eternal conditioned Brownian snake,
we easily get, for every $\lambda>0$,
\begin{align}
\label{LaplaceI}
E \left[\exp - \lambda \mathcal{I} \left( [0,1] \right) \right]
& = E \Big[\exp - \frac{\lambda}{2} \int_0^{\infty} \mathrm{d}s \, \Big( \mathbf{1}
_{\{ \widehat{W}^{(L)}_s \leq 1 \}} + 
\mathbf{1}_{\{ \widehat{W}^{(R)}_s \leq 1 \}} \Big)  \Big] \notag\\
& = E \Big[\exp - 4 \int_0^{\infty} \mathrm{d}t \, \mathbb{N}_{Z_t}
\Big( \mathbf{1}_{\{ \mathcal{R} \subset ]0,\infty[ \}}
( 1 - \exp - \frac{\lambda}{2} \int_0^{\sigma} \mathrm{ds} \,
\mathbf{1}_{\{ \widehat{W}_s \leq 1 \}} ) \Big) \Big].
\end{align}
Using (\ref{range-estimate}), formula \eqref{LaplaceI} can be rewritten as
\begin{equation}
\label{Laplaceoccup}
E \left[\exp - \lambda \mathcal{I} \left( [0,1] \right) \right]
= E \left[\exp - 4 \int_0^{\infty} \mathrm{dt} \, \left( u_{\lambda/2}(Z_t) -  \frac{3}{2 Z_t^2} \right) \right]
\end{equation}
where for every $\lambda, x >0$,
\[
u_{\lambda}(x) = \mathbb{N}_x \left( 1 - \mathbf{1}_{\left\{ \mathcal{R} \subset ]0,\infty[ \right\}}
\exp - \lambda \int_0^{\sigma} \mathrm{ds} \, \mathbf{1}_{\left\{ \widehat{W}_s \leq 1 \right\}} \right).
\]
From the known connections between the Brownian snake and
partial differential equations (see Chapters V and VI of the monograph \cite{LGsnake})
or by adapting the proof of Lemma 6 in \cite{Delmas}, one checks that the
the function $u_{\lambda}$ is monotone decreasing and continuously differentiable on $]0, \infty[$, and solves the differential equation
$$
\frac{1}{2} u'' = 2 u^2 - \lambda
$$
in $]0,1[$, with the boundary condition $u_\lambda(0) = \infty$,
and the equation
$$\frac{1}{2} u'' = 2 u^2$$
in $]1,\infty[$. From these equations, one can derive analytic formulas for $u_\lambda$.
Still it does not seem easy to use these formulas in order to compute the 
Laplace transform \eqref{Laplaceoccup}. We content ourselves with a
first moment calculation.

\begin{proposition}
\label{firstmom}
For every nonnegative measurable function $g$ on ${\mathbb R}_+$,
$$E[\langle{\mathcal I},g\rangle]= \frac{128}{21} \int_0^\infty dr\,r^{3}\,g(r).$$
In particular, for every $r>0$,
$$E[{\mathcal I}([0,r])]= \frac{32}{21}\,r^4.$$
\end{proposition}

\begin{proof}
From the definition of $\mathcal I$ and the construction of the eternal 
conditioned Brownian snake, we get
$$E[\langle{\mathcal I},g\rangle]
= 4\,E \left[ \int_0^{\infty} \mathrm{d}t \, \mathbb{N}_{Z_t} \left(
\mathbf{1}_{\left\{ \mathcal{R} \subset ]0,\infty[ \right\}}
\int_0^{\sigma} \mathrm{d}s \, g(\widehat{W}_s )
\right) \right].$$
For every $z>0$, let
\[
\varphi_g(z)  = \mathbb{N}_{z} \left(
\mathbf{1}_{\left\{ \mathcal{R} \subset ]0,\infty[ \right\}}
\int_0^{\sigma} \mathrm{d}s \, g(\widehat{W}_s ).
\right)
\]
Let $(\xi_t)_{t \geq 0}$ denote a linear Brownian motion that starts from $z$ under the probability measure $P_z$.
Then, by the case $p=1$ of Theorem 2.2 in \cite{LGW}, we have
\begin{align*}
\varphi_g(z) & = \int_0^{\infty} \mathrm{d}a \,
E_z \left[
g(\xi_a)\,
\exp \left( -4 \int_0^a \mathrm{d}s \, \mathbb{N}_{\xi_s} \left(
  \mathcal{R} \cap ]0, \infty [ \neq \emptyset \right) \right)
\right]\\
& = \int_0^{\infty} \mathrm{d}a \,
E_z \left[
g(\xi_a)\,
\exp \left( -6 \int_0^a \frac{\mathrm{d}s}{\xi_s^2} \right) \right]\\
& = \int_0^{\infty} \mathrm{d}a \, z^4\,
E_z\left[Z_a^{-4} g(Z_a) \right],
\end{align*}
where the nine-dimensional Bessel process $Z$ starts from $z$ under the probability measure
$P_z$. In the second equality we used (\ref{range-estimate}), and in the third one we applied the absolute continuity properties of laws
of Bessel processes (see e.g. Proposition 2.6 in \cite{LGW}). 

Recall that the nine-dimensional Bessel process has the same distribution as the
Euclidean norm of a nine-dimensional Brownian motion. Using the
explicit form of the Green function of Brownian motion in ${\mathbb R}^9$, we get
$$
\varphi_g(z) = 2\pi^{9/2}\,\Gamma(\frac{7}{2})\,z^4\,\int_{{\mathbb R}^9}
{\mathrm d}y\,|y-x_z|^{-7}\,|y|^{-4}\,g(|y|),$$
where $x_z$ is an arbitrary point of ${\mathbb R}^9$ such that $|x_z|=z$. 
For every $r>0$, let $\sigma_r(dy)$ be the uniform probability measure
on the sphere of radius $r$ centered at the origin in ${\mathbb R}^9$.
Since the function $y\mapsto |y-x_z|^{-7}$ is harmonic, an easy argument
gives
$$\int \sigma_r(dy)\,|y-x_z|^{-7}= (r\vee z)^{-7}.$$
We can then integrate in polar coordinates in the previous formula for $\varphi_g(z)$,
and get
$$\varphi_g(z) = \frac{8}{7}\,z^4\,\int_0^\infty dr\,r^4\,(r\vee z)^{-7}\,g(r).$$
By substituting this in the first display of the proof, and arguing in a similar way
as above, we obtain
\begin{align*}
E[\langle{\mathcal I},g\rangle]
&= 4\,E \Big[ \int_0^{\infty} \mathrm{d}t \,\varphi_g(Z_t)\Big]\\
&= 4 \Big(\frac{8}{7}\Big)^2\,\int_0^\infty {\mathrm d}z\,z^5\,
\int_0^\infty {\mathrm d}r\,r^4\,(r\vee z)^{-7}\,g(r)\\
&=4 \Big(\frac{8}{7}\Big)^2\,\int_0^\infty dr\,r^4\,g(r) \Big(
r^{-7}\int_0^r {\mathrm d}z\,z^5 + \int_r^\infty {\mathrm d}z\,z^{-2}\Big)\\
&=\frac{128}{21}\,\int_0^\infty {\mathrm d}r\,r^3\,g(r).
\end{align*}
This completes the proof of Proposition \ref{firstmom}. 
\end{proof}

\addcontentsline{toc}{section}{References}
\bibliographystyle{abbrv}
\bibliography{scaling-infinite}

\def\cprime{$'$}
\begin{thebibliography}{10}

\bibitem{Aldous}
D.~Aldous.
\newblock The continuum random tree. {III}.
\newblock {\em Ann. Probab.}, 21(1):248--289, 1993.

\bibitem{2DQG}
J.~Ambj{\o}rn, B.~Durhuus, and T.~Jonsson.
\newblock {\em Quantum geometry}.
\newblock Cambridge Monographs on Mathematical Physics. Cambridge University
  Press, Cambridge, 1997.
\newblock A statistical field theory approach.

\bibitem{A1}
O.~Angel.
\newblock Growth and percolation on the uniform infinite planar triangulation.
\newblock {\em Geom. Funct. Anal.}, 13(5):935--974, 2003.

\bibitem{AS}
O.~Angel and O.~Schramm.
\newblock Uniform infinite planar triangulations.
\newblock {\em Comm. Math. Phys.}, 241(2-3):191--213, 2003.

\bibitem{BIPZ}
E.~Br{\'e}zin, C.~Itzykson, G.~Parisi, and J.~B. Zuber.
\newblock Planar diagrams.
\newblock {\em Comm. Math. Phys.}, 59(1):35--51, 1978.

\bibitem{CD}
P.~Chassaing and B.~Durhuus.
\newblock Local limit of labeled trees and expected volume growth in a random
  quadrangulation.
\newblock {\em Ann. Probab.}, 34(3):879--917, 2006.

\bibitem{CS}
P.~Chassaing and G.~Schaeffer.
\newblock Random planar lattices and integrated super{B}rownian excursion.
\newblock {\em Probab. Theory Related Fields}, 128(2):161--212, 2004.

\bibitem{CV}
R.~Cori and B.~Vauquelin.
\newblock Planar maps are well labeled trees.
\newblock {\em Canad. J. Math.}, 33(5):1023--1042, 1981.

\bibitem{Delmas}
J.-F. Delmas.
\newblock Computation of moments for the length of the one dimensional {ISE}
  support.
\newblock {\em Electron. J. Probab.}, 8:no. 17, 15 pp. (electronic), 2003.

\bibitem{JM}
S.~Janson and J.-F. Marckert.
\newblock Convergence of discrete snakes.
\newblock {\em J. Theoret. Probab.}, 18(3):615--647, 2005.

\bibitem{Kr}
M.~Krikun.
\newblock Local structure of random quadrangulations.
\newblock {\em Preprint}, \url{http://arxiv.org/abs/math/0512304}, 2005.

\bibitem{Krikun}
M.~Krikun.
\newblock A uniformly distributed infinite planar triangulation and a related
  branching process.
\newblock {\em J. Math. Sci. (N.Y.)}, 131(2):5520--5537, 2005.

\bibitem{LGsnake}
J.-F. Le~Gall.
\newblock {\em Spatial branching processes, random snakes and partial
  differential equations}.
\newblock Lectures in Mathematics ETH Z\"urich. Birkh\"auser Verlag, Basel,
  1999.

\bibitem{randomtrees}
J.-F. Le~Gall.
\newblock Random trees and applications.
\newblock {\em Probab. Surv.}, 2:245--311 (electronic), 2005.

\bibitem{LG06}
J.-F. Le~Gall.
\newblock A conditional limit theorem for tree-indexed random walk.
\newblock {\em Stochastic Process. Appl.}, 116(4):539--567, 2006.

\bibitem{LG}
J.-F. Le~Gall.
\newblock The topological structure of scaling limits of large planar maps.
\newblock {\em Invent. Math.}, 169(3):621--670, 2007.

\bibitem{LGW}
J.-F. Le~Gall and M.~Weill.
\newblock Conditioned {B}rownian trees.
\newblock {\em Ann. Inst. H. Poincar\'e Probab. Statist.}, 42(4):455--489,
  2006.

\bibitem{MaMo}
J.-F. Marckert and A.~Mokkadem.
\newblock Limit of normalized quadrangulations: The brownian map.
\newblock {\em Ann. Probab.}, 34(6):2144--2202, 2006.

\bibitem{Me}
L.~M\'enard.
\newblock The two uniform infinite quadrangulations of the plane have the same
  law.
\newblock {\em Ann. Inst. H. Poincar\'e Probab. Statist.}, 46(1):190--208,
  2010.

\bibitem{RY}
D.~Revuz and M.~Yor.
\newblock {\em Continuous martingales and {B}rownian motion}, volume 293 of
  {\em Grundlehren der Mathematischen Wissenschaften}.
\newblock Springer-Verlag, Berlin, third edition, 1999.

\bibitem{Schaeffer}
G.~Schaeffer.
\newblock {\em Conjugaisons d'arbres et cartes combinatoires al\'eatoires}.
\newblock {P}h{D} thesis, Universit\'e de Bordeaux I, 1998.

\bibitem{Tutte}
W.~T. Tutte.
\newblock A census of planar maps.
\newblock {\em Canad. J. Math.}, 15:249--271, 1963.

\end{thebibliography}
\end{document}